\documentclass[11pt]{article}
\usepackage[T1]{fontenc}
\usepackage[margin=0.9in]{geometry}
\usepackage{amsmath,amssymb,amsthm,mathtools,microtype}
\usepackage[title,titletoc]{appendix}
\usepackage{graphicx,tikz}
\usetikzlibrary{arrows.meta,calc,positioning}
\usepackage[authoryear,round]{natbib}
\usepackage[linktocpage,colorlinks,linkcolor=blue,anchorcolor=blue,citecolor=blue,urlcolor=blue,pagebackref]{hyperref}
\hypersetup{pdftitle={Heavy-Ball Method under Randomized Schedules},pdfauthor={Chang He and Shuzhong Zhang},pdfsubject={Deterministic and randomized time boundaries; fixed-time and anytime convergence}}
\allowdisplaybreaks[2]
\newtheorem{theorem}{Theorem}
\newtheorem{corollary}[theorem]{Corollary}
\newtheorem{lemma}[theorem]{Lemma}

\newcommand{\R}{\mathbb{R}}
\newcommand{\E}{\mathbb{E}}

\newcommand{\inner}[2]{\langle#1,#2\rangle}

\newcommand{\cF}{\mathcal{F}}
\newcommand{\dd}{\,\mathrm{d}}
\newcommand{\cO}{\mathcal{O}}
\newcommand{\ee}{\mathrm{e}}
\title{Heavy-Ball Method under Randomized Schedules}
\author{
{Chang He} \thanks{Department of Industrial and Systems Engineering, University of Minnesota. \texttt{hec@umn.edu}}
\and
{Shuzhong Zhang} \thanks{Department of Industrial and Systems Engineering, University of Minnesota. \texttt{zhangs@umn.edu}}
}
\date{}

\begin{document}
\maketitle
\begin{abstract}
We study how predefined randomized parameter schedules accelerate the heavy-ball method on general smooth convex objectives. Our analysis distinguishes two levels of randomization: sampling gradient-evaluation times within intervals whose boundaries are deterministic, and additionally randomizing the time boundaries themselves. With deterministic time boundaries, we construct fixed-time and anytime schedules that achieve an expected last-iterate function value gap of order $\mathcal{O}(1/K^{4/3})$; the anytime schedule also satisfies the same rate almost surely. We then use randomized time boundaries and obtain the improved last-iterate rate $\mathcal{O}(1/K^{3/2})$, both in expectation and almost surely. To the best of our knowledge, this is the first global nonasymptotic convergence guarantee for the heavy-ball method on general smooth convex functions that improves polynomially over the classical $\mathcal{O}(1/K)$ rate. Our result shows that the heavy-ball method achieves a strictly better convergence rate than the best known result $\mathcal{O}(1/K^{\log_2(1+\sqrt{2})})$ attainable by plain gradient descent with silver stepsize schedules.
\end{abstract}

\section{Introduction}\label{sec:introduction}
\subsection{Setup and main results}
Consider the unconstrained convex minimization problem
\[
 f^\star\coloneqq\min_{x\in\R^d} f(x),
\]
where $f\colon\R^d\to\R$ is convex, attains its minimum at $x^\star$, and has an $L$-Lipschitz gradient, with $L>0$. We study the heavy-ball method \citep{polyak1964some}
\begin{equation}\label{eq:hb}
 x^{k+1}=x^k-\eta_k\nabla f(x^k)+\beta_k(x^k-x^{k-1}),
 \  x^{-1}=x^0,\quad\beta_0=0.
\end{equation}
The gradient oracle is exact; only the parameter schedule is randomized. A schedule is \emph{predefined} if it can be generated from the iteration index, the smoothness constant $L$, and an oracle-independent random seed, without using observed function values, gradients, or iterates. A fixed-time schedule additionally depend on a prescribed total iteration count $K$. An anytime schedule is a single infinite sequence, independent of a terminal iteration, for which the stated guarantee holds at every deterministic $K\ge1$. In either setting, $x^K$ is the last iterate after exactly $K$ gradient evaluations. We organize the construction around increasing time boundaries $A_0=1<A_1<\cdots$ and evaluation times $u_k\in[A_k,A_{k+1}]$. There are two distinct sources of randomness. First, the evaluation times inside each interval can be sampled at random when the boundaries are deterministic. Second, the boundaries themselves can be sampled at random before any oracle evaluations.

\paragraph{Deterministic time boundaries.}
For a prescribed total iteration count $K\ge1$, choose the deterministic boundaries $A_k=(1+k/6K^{1/3})^2$ for all $0 \le k \le K$. Independently sample $u_k\sim\operatorname{Unif}(A_k,A_{k+1})$ for $0\le k<K$, and set $u_K=A_K$. Set the coefficients as
\begin{equation}\label{eq:fixed-coefficients}
\begin{aligned}
 \eta_k=\frac{A_{k+1}-A_k}{L}\Big(1-\frac{u_k}{u_{k+1}}\Big), \ 0\le k<K, \ \text{and} \ \beta_k =\frac{u_k^{-1}-u_{k+1}^{-1}}{u_{k-1}^{-1}-u_k^{-1}}, \ 1\le k<K, \ \beta_0 = 0.
\end{aligned}
\end{equation}
The heavy-ball method satisfies
\begin{equation*}
    \mathbb E[f(x^K)-f^\star]\le\frac{36\ee^{1/7}L\|x^0-x^\star\|^2}{K^{4/3}}.
\end{equation*}
We then strengthen the interpolation and potential to obtain an anytime guarantee with the same rate. Choose the deterministic boundaries $A_k=(1+k/12)^{4/3}$ and independently sample $u_k\sim\operatorname{Unif}(A_k,A_{k+1})$ for every $k\ge0$. Set the coefficients as
\begin{equation}\label{eq:convex-coefficients}
 \eta_k=\frac{A_{k+1}-A_k}{2L}\Big(1-\frac{u_k^2}{u_{k+1}^2}\Big), \ k\ge0, \ \text{and} \ \beta_k=\frac{u_k^{-2}-u_{k+1}^{-2}}{u_{k-1}^{-2}-u_k^{-2}}, \ k\ge1, \ \beta_0 = 0.
\end{equation}
Then, for every deterministic $K\ge1$, the heavy-ball method satisfies
\begin{equation*}
 \E[f(x^K)-f^\star]
 \le\frac{84L\|x^0-x^\star\|^2}{K^{4/3}}.
\end{equation*}
Moreover, $f(x^K)-f^\star=\mathcal O(K^{-4/3})$ almost surely. The boundaries in these two constructions are deterministic, but the heavy-ball coefficients are randomized through the sampled evaluation times.

\paragraph{Randomized time boundaries.} The result improves the anytime rate by randomizing the boundaries as well. Independently sample $V_k\sim\operatorname{Unif}[1,2]$ and $U_k\sim\operatorname{Unif}[0,1]$, with all these variables mutually independent. Define
\begin{equation}\label{eq:random-grid}
 A_0=1,\  A_{k+1}=A_k+\frac{A_k^{1/3}}{512}V_k,
 \  k\ge0.
\end{equation}
Thus $A_k$ is random for $k\ge1$, and $A_0$ remains fixed. The boundaries and all coefficients are still predefined, because their construction is independent of the oracle. For every $k\ge0$, define the evaluation time by
\begin{equation}\label{eq:random-sampling}
 u_k=A_k+(A_{k+1}-A_k)U_k.
\end{equation}
Use the same coefficient formula as in the deterministic-boundary anytime construction:
\begin{equation}\label{eq:random-coefficients}
 \eta_k=\frac{A_{k+1}-A_k}{2L}\Big(1-\frac{u_k^2}{u_{k+1}^2}\Big), \ k\ge0, \ \text{and} \ \beta_k=\frac{u_k^{-2}-u_{k+1}^{-2}}{u_{k-1}^{-2}-u_k^{-2}}, \ k\ge1, \ \beta_0 = 0.
\end{equation}
Theorem~\ref{thm:random-anytime} shows that, for every deterministic $K\ge1$, the heavy-ball method satisfies
\begin{equation*}
 \E[f(x^K)-f^\star] \le\frac{131072L\|x^0-x^\star\|^2}{K^{3/2}}.
\end{equation*}
The same infinite schedule also satisfies
\begin{equation*}
 \mathbb P \Big(\sup_{K\ge1}K^{3/2}[f(x^K)-f^\star]<\infty\Big)=1.
\end{equation*}
In particular, $f(x^K)-f^\star=\mathcal O(K^{-3/2})$ almost surely. For all constructions, the coefficient denominators are positive almost surely. We remark that, compared with the conventional choice of momentum coefficients, the momentum coefficients under a randomized schedule need not be smaller than one and can be arbitrarily large.

\subsection{Related literature}

\paragraph{Stepsize schedules for accelerating gradient descent.}
For smooth convex optimization, gradient descent with a suitable constant stepsize achieves the classical $\mathcal{O}(1/K)$ rate \citep{drori2014performance}, whereas Nesterov's accelerated method attains the optimal $\mathcal{O}(1/K^2)$ rate \citep{nesterov1983method,nemirovsky1983problem}. Using performance-estimation techniques, \citet{grimmer2024long} shows that schedules containing steps larger than $2/L$ can improve convergence constants. Polynomial improvements through nonconstant stepsizes were established by \citet{grimmer2023accelerated,altschuler2023silver}, with the silver stepsize schedule achieving $\mathcal{O}(K^{-r})$ for fixed-time convergence, where $r=\log_2(1+\sqrt{2})\approx1.2716$. Subsequent work improves the constants and develops further schedules for fixed-time convergence \citep{grimmer2024accelerated,grimmer2025composing,zhang2026concatenation}. For anytime convergence, \citet{zhang2025anytime} establishes an $\mathcal{O}(K^{-2r/(1+r)})$ convergence rate, with exponent approximately $1.1195$. Complementary work establishes lower bounds on acceleration by deterministic, nonadaptive stepsize schedules \citep{tsai2026lower,ma2026lower,ye2026improved,jung2026stronger}.

\paragraph{Heavy-ball methods for convex optimization.}
The heavy-ball method was introduced by \citet{polyak1964some}, and its acceleration properties are best understood for quadratic objectives. For strongly convex quadratics, appropriately chosen constant parameters achieve the optimal asymptotic dependence on the condition number \citep{nemirovski1995information}. For general convex quadratics, suitable time-varying parameters yield an $\mathcal{O}(1/K^2)$ convergence rate \citep{flammarion2015averaging,orvieto2025lyapunov}. Such acceleration does not automatically extend beyond quadratics: \citet{lessard2016analysis} exhibit a smooth strongly convex example for which parameters optimized for quadratics generate a nonconvergent cycle, while \citet{goujaud2025nonaccelerations} show that no fixed choice of stepsize and momentum can attain an accelerated worst-case rate over smooth strongly convex functions. For general smooth convex objectives, \citet{ghadimi2015global} establish an $\mathcal{O}(1/K)$ rate, including an individual-iterate guarantee with time-varying parameters, while \citet{sun2019nonergodic} obtain a non-ergodic rate of the same order under coercivity. With a suitable time-varying parameter schedule, \citet{sebbouh2021almost} further obtain an asymptotic last-iterate rate of $o(1/K)$.

\section{Preliminaries}\label{sec:preliminaries}
\subsection{Notation and terminology}\label{subsec:notation}
We adopt the following notation throughout this paper. Let $\mathbb{R}^d$ denote the $d$-dimensional Euclidean space. The standard inner product is denoted by $\langle \cdot, \cdot \rangle$. We use $\|\cdot\|$ to denote the Euclidean norm for vectors. We say that the objective function $f$ has an $L$-Lipschitz gradient, or that it is $L$-smooth, if it satisfies
\begin{align*}
    \|\nabla f(x) - \nabla f(y)\| \le L \|x-y\|, \ \forall x, y \in \R^d.
\end{align*}
Under the Lipschitz continuity of the gradient, the objective function satisfies the following descent inequality:
\begin{equation*}
    \left| f(y) - f(x) - \langle \nabla f(x), y - x \rangle \right| \le \frac{L}{2} \|x - y\|^2, \  \forall x, y \in \mathbb{R}^d,
\end{equation*}
and the self-bounding inequality
\begin{equation}\label{eq:self-bounding-inequality}
    \|\nabla f(x)\|^2 \le 2L(f(x) - f^\star), \  \forall x \in \mathbb{R}^d.
\end{equation}
The above two inequalities imply the following Lipschitz property of the square-root optimality gap:
\begin{equation}\label{eq:Lipschitz-square-root-optimality-gap}
    \big|\sqrt{f(x) - f^\star} - \sqrt{f(y)-f^\star}\big| \le \sqrt{\frac{L}{2}}\|x - y\|, \  \forall x, y \in \mathbb{R}^d.
\end{equation}
The proof follows from the observation that
\begin{align*}
    f(y)-f^\star \le\ & f(x)-f^\star+\langle\nabla f(x),y-x\rangle+\frac L2\|y-x\|^2\\
    \le\ & f(x)-f^\star+\sqrt{2L(f(x)-f^\star)}\|y-x\|+\frac L2\|y-x\|^2\\
    =\ & \big(\sqrt{f(x)-f^\star}+\sqrt{\frac L2}\|y-x\|\big)^2.
\end{align*}
Furthermore, when the objective function is also convex, the following smooth convex interpolation inequality holds:
\begin{equation}\label{eq:cocoercivity-inequality}
     f(y) - f(x) - \langle \nabla f(x), y - x \rangle \ge \frac{1}{2L} \|\nabla f(x) - \nabla f(y)\|^2, \  \forall x, y \in \mathbb{R}^d.
\end{equation}
These standard smoothness and interpolation inequalities are given in Lemma~1.2.3 and Theorem~2.1.5 of \citet{nesterov2018lectures}. 

Let $m^k \coloneqq x^k-x^{k-1}$ denote the momentum at iteration $k$, with $m^0=0$. Then the heavy-ball method \eqref{eq:hb} can be written as the two-step update
\begin{align*}
m^{k+1} = \ &\beta_k m^k-\eta_k\nabla f(x^k),\\
x^{k+1} = \ &x^k+m^{k+1}.
\end{align*}
However, the momentum is rescaled at every iteration by the multiplicative factor $\beta_k$. To remove this multiplicative rescaling, we introduce a positive sequence $\{\delta_k\}$ satisfying $\beta_k=\delta_k/\delta_{k-1}$, $k \ge 1$, and define the \textit{rescaled momentum} as $p^k\coloneqq m^k/\delta_{k-1}$, $k \ge 1$ with $p^0=0$. The heavy-ball method can then be equivalently written as
\begin{equation}\label{eq:rescaled-two-step}
\begin{aligned}
p^{k+1} = \ &p^k-\frac{\eta_k}{\delta_k}\nabla f(x^k),\\
x^{k+1} = \ &x^k+\delta_k p^{k+1}.
\end{aligned}
\end{equation}
In the convergence analysis, we work directly with the update scheme \eqref{eq:rescaled-two-step}. 

\subsection{Technique overview} 
For deterministic boundaries, we use $\cF_k=\sigma(u_0,\ldots,u_{k-1})$, with $\cF_0=\{\varnothing,\Omega\}$. For randomized boundaries, the filtration is instead
\begin{equation}\label{eq:random-filtration}
 \mathcal G_k=\sigma(V_0,U_0,\ldots,V_{k-1},U_{k-1}),
 \ \mathcal G_0=\{\varnothing,\Omega\}.
\end{equation}
In the latter case, $A_k$ is $\mathcal G_k$-measurable, and, conditional on $\mathcal G_k$ and $V_k$, the interval $[A_k,A_{k+1}]$ is fixed and $u_k$ is uniform on it. The evaluation times $u_k$ are not mutually independent when the boundaries are random. All uniform-sampling identities below are applied only at gradient-evaluation indices. Each schedule uses an incoming curve $\xi_{k-1}$ and an outgoing curve $\xi_k$ satisfying
\[
 \xi_{k-1}(u_k)=x^k=\xi_k(u_k),\ 
 x^{k+1}=\xi_k(u_{k+1}),\  \xi_{-1}(s)=x^0.
\]
The outgoing curve moves in the direction $p^{k+1}$. Its parameterization determines $\delta_k$ and the form of the potential function. Figure~\ref{fig:curves} illustrates the geometric interpretation behind both schedules. The incoming curve and $p^k$ are measurable with respect to $\cF_k$ in the deterministic-boundary setting and $\mathcal G_k$ in the randomized-boundary setting. By contrast, $x^k$ becomes measurable only after its evaluation time $u_k$ is sampled. For an integrable function $\psi$ fixed under the indicated conditioning, uniform sampling gives
\begin{equation*}
 \int_{A_k}^{A_{k+1}}\psi(s)\dd s
 =\begin{cases}
 (A_{k+1}-A_k)\E[\psi(u_k)\mid\cF_k],
   &\text{deterministic boundaries},\\
 (A_{k+1}-A_k)\E[\psi(u_k)\mid\mathcal G_k,V_k],
   &\text{randomized boundaries}.
 \end{cases}
\end{equation*}
The required random variables can be sampled before the current gradient evaluation. In the rescaled representation \eqref{eq:rescaled-two-step}, $p^{k+1}$ and $\xi_k$ depend only on the current and previous gradient-evaluation times, not on $u_{k+1}$.

\begin{figure}[t]
 \centering
 \includegraphics[width=\textwidth]{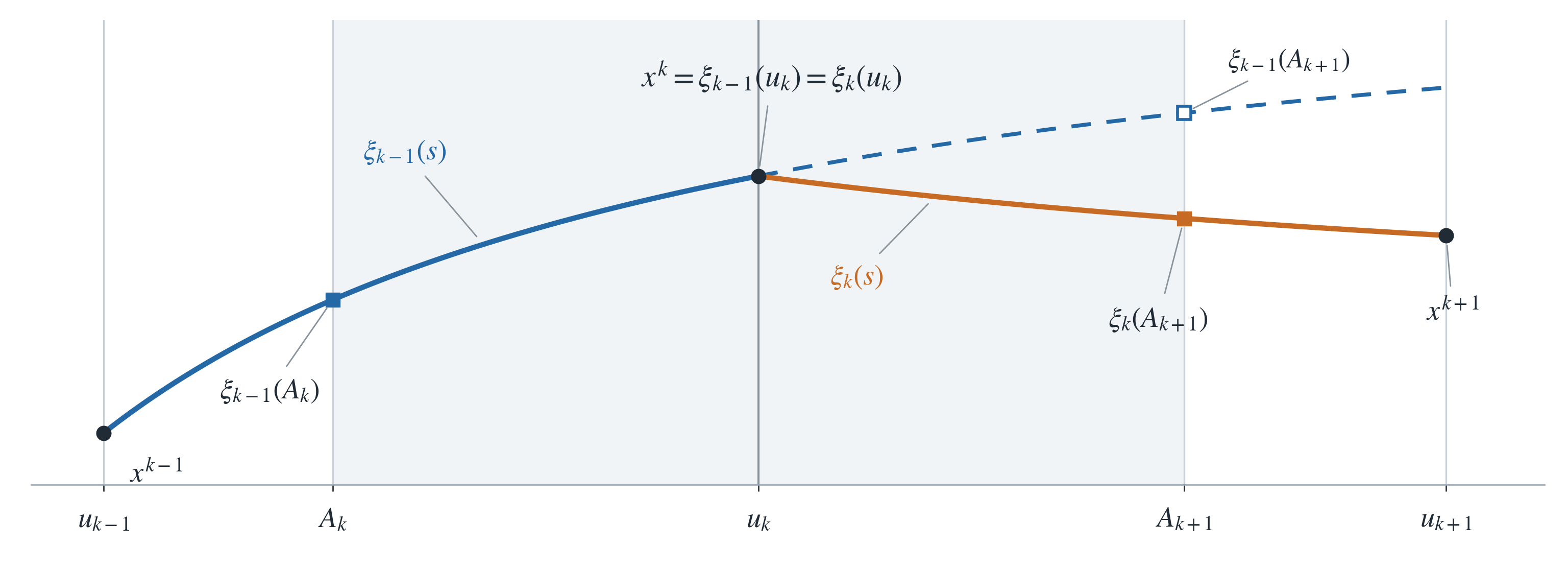}
 \caption{The common interpolation picture. Black circles are iterates at sampled times, and squares are curve states at interval boundaries. The blue dashed segment continues the incoming curve without the gradient update at $u_k$, while the orange segment follows the outgoing curve.}
 \label{fig:curves}
\end{figure}

\section{Deterministic Time Boundaries: Fixed-Time Convergence}\label{sec:fixed-proof}
We first keep the boundaries deterministic and prescribe a total iteration count $K$. Only the evaluation times $u_k$ inside the intervals are randomized. The construction in this section proves the fixed-time $\cO(1/K^{4/3})$ guarantee.

\subsection{Schedule and interpolation} Fix an iteration $k$ and consider the update after the gradient at $x^k$ is evaluated and the rescaled momentum is updated to $p^{k+1}$. Note that $x^{k+1}$ is updated from $x^k$ along the direction $p^{k+1}$ in \eqref{eq:rescaled-two-step}. We represent this discrete
iterate update by an auxiliary twice-differentiable curve $\xi_k(s)$ starting from $x^k$ at time $u_k>0$, i.e., $\xi_k(u_k)=x^k$, and moving along the direction $p^{k+1}$. Along this curve, we have
\begin{equation}\label{eq:integral}
\begin{aligned}
    \frac{\dd}{\dd s}[s(f(\xi_k(s))-f^\star)]
    =\ & f(\xi_k(s))-f^\star
        +s\langle\nabla f(\xi_k(s)),\xi_k'(s)\rangle\\
    \le\ & \langle\nabla f(\xi_k(s)),
        \xi_k(s)+s\xi_k'(s)-x^\star\rangle.
\end{aligned}
\end{equation}
where the inequality follows from convexity. We choose the parameterization so that $\xi_k(s)+s\xi_k'(s)$ is constant along the curve, i.e., $\frac{\dd}{\dd s}(\xi_k(s)+s\xi_k'(s))=0$. Since the curve $\xi_k(s)$ moves along the fixed direction $p^{k+1}$, there exists a scalar function $a_k(s)$ such that $\xi_k'(s)=a_k(s)p^{k+1}$. Consequently, it follows that
\begin{align*}
    \frac{\dd}{\dd s}\bigl(\xi_k(s)+s\xi_k'(s)\bigr)
    =2\xi_k'(s)+s\xi_k''(s)
    =\bigl(2a_k(s)+sa_k'(s)\bigr)p^{k+1}.
\end{align*}
It therefore suffices to choose $sa_k'(s)+2a_k(s)=0$. We normalize the direction by choosing the solution $a_k(s)=s^{-2}$. Integrating $\xi_k'(s)=p^{k+1}/s^2$ from $u_k$ yields
\begin{equation*}
    \xi_k(s)=x^k+\Big(\frac{1}{u_k}-\frac{1}{s}\Big)p^{k+1}, \  \forall s\ge u_k.
\end{equation*}

For a iteration $0\le k<K-1$, we next choose where to evaluate the next gradient along the curve $\xi_k(\cdot)$, i.e., we determine the time $u_{k+1} \in [A_{k+1},A_{k+2}]$. Integrating \eqref{eq:integral} over the interval $[A_{k+1},A_{k+2}]$ gives
\begin{align*}
   A_{k+2}[f(\xi_k(A_{k+2}))-f^\star]-A_{k+1}[f(\xi_k(A_{k+1}))-f^\star] \le \int_{A_{k+1}}^{A_{k+2}} \langle\nabla f(\xi_k(s)), x^k+\frac{p^{k+1}}{u_k}-x^\star\rangle\dd s,
\end{align*}
where we have substituted $\xi_k(s)+s\xi_k'(s)=\xi_k(u_k)+u_k\xi_k'(u_k)=x^k+\frac{p^{k+1}}{u_k}$, as $\xi_k(s)+s\xi_k'(s)$ is constant with respect to $s$. Since the heavy-ball method uses only one gradient evaluation at each iteration, we sample $u_{k+1}\sim\operatorname{Unif}(A_{k+1},A_{k+2})$, independently of the previously sampled times, and update
\begin{equation}\label{eq:fixed-update-x}
     x^{k+1}=\xi_k(u_{k+1})=x^k+\Big(\frac{1}{u_k}-\frac{1}{u_{k+1}}\Big)p^{k+1}.
\end{equation}
Comparing the above identity with \eqref{eq:rescaled-two-step} gives $ \delta_k=\frac{1}{u_k}-\frac{1}{u_{k+1}}$. By the uniform-sampling identity, we obtain
\begin{align*}
    \int_{A_{k+1}}^{A_{k+2}} \langle\nabla f(\xi_k(s)), x^k+\frac{p^{k+1}}{u_k}-x^\star\rangle\dd s
    = (A_{k+2}-A_{k+1})\mathbb E[\langle\nabla f(x^{k+1}),x^{k+1}+\frac{p^{k+1}}{u_{k+1}}-x^\star\rangle \mid \cF_{k+1}],
\end{align*}
since $x^{k+1} = \xi_k(u_{k+1})$, and $\xi_k$, $p^{k+1}$ are $\cF_{k+1}$-measurable. We now determine the next stepsize. At $x^{k+1}$, the rescaled momentum update and \eqref{eq:fixed-update-x} imply
\begin{equation*}
    x^{k+1}+\frac{p^{k+2}}{u_{k+1}}-x^\star=x^k+\frac{p^{k+1}}{u_k}-x^\star-\frac{\eta_{k+1}}{u_{k+1}\delta_{k+1}}\nabla f(x^{k+1}).
\end{equation*}
Squaring both sides and rearranging yields
\begin{align*}
    &\langle\nabla f(x^{k+1}), x^{k+1}+\frac{p^{k+1}}{u_{k+1}}-x^\star\rangle - \frac{\eta_{k+1}}{2u_{k+1}\delta_{k+1}}
      \|\nabla f(x^{k+1})\|^2\\
    = \ &\frac{u_{k+1}\delta_{k+1}}{2\eta_{k+1}}[\|x^k+\frac{p^{k+1}}{u_k}-x^\star\|^2-\|x^{k+1}+\frac{p^{k+2}}{u_{k+1}}-x^\star\|^2]
\end{align*}
To obtain a telescoping difference of distances, we set
\begin{align*}
   \frac{(A_{k+2}-A_{k+1})u_{k+1}\delta_{k+1}}{2\eta_{k+1}}=\frac{L}{2},
\end{align*}
which yields the selection of the stepsize.

Under the above development, for every $k=0,\ldots,K-1$, independently sample $u_k\sim\operatorname{Unif}(A_k,A_{k+1})$, and set $u_K=A_K$. Use the schedule
\begin{equation}\label{eq:fixed-parameters}
    \delta_k=\frac1{u_k}-\frac1{u_{k+1}} > 0, \ \eta_k=\frac{(A_{k+1}-A_k)u_k\delta_k}{L}, \ \beta_k=\frac{\delta_k}{\delta_{k-1}}\ (k\ge1).
\end{equation}
All $\delta_k$ are positive almost surely. Define $\xi_{-1}(s)=x^0$ for $s\ge A_0$. For all iterates, we have $x^k=\xi_{k-1}(u_k)$ for $0\le k<K$, and $x^K=\xi_{K-1}(A_K)$. All these parameters can be sampled before the algorithm starts, and they give exactly the coefficients in \eqref{eq:fixed-coefficients}. The rescaled momentum update simplifies to
\begin{equation}\label{eq:fixed-rescaled-momentum}
    p^{k+1}=p^k-\frac{(A_{k+1}-A_k)u_k}{L}\nabla f(x^k).
\end{equation}

\subsection{Potential function decrease} 
We introduce the potential function over the interval $[A_k,A_{k+1}]$ as follows:
\begin{equation}\label{eq:fixed-energy}
    \Phi_k=A_k[f(\xi_{k-1}(A_k))-f^\star]+\frac{L}{2}\|\xi_{k-1}(A_k)+\frac{p^k}{A_k}-x^\star\|^2,
\end{equation}
which combines the function value gap and the optimality distance along the curve. For $k\ge1$, recall that
\begin{align*}
    \xi_{k-1}(s)= \ &x^{k-1}+\Big(\frac{1}{u_{k-1}}-\frac{1}{s}\Big)p^k,\\
    x^k=\xi_{k-1}(u_k)= \ &x^{k-1}+\Big(\frac{1}{u_{k-1}}-\frac{1}{u_k}\Big)p^k.
\end{align*}
Subtracting the second identity from the first gives $\xi_{k-1}(s)=x^k+(\frac{1}{u_k}-\frac{1}{s})p^k$. Since $\xi_k(s)=x^k+(\frac{1}{u_k}-\frac{1}{s})p^{k+1}$, the difference between the two curves satisfies
\begin{align*}
    \xi_k(s)-\xi_{k-1}(s)= \Big(\frac{1}{u_k}-\frac{1}{s}\Big)(p^{k+1}-p^k) = -\frac{(A_{k+1}-A_k)(s-u_k)}{Ls}\nabla f(x^k), \ \forall s\ge u_k,
\end{align*}
where the second equality follows from \eqref{eq:fixed-rescaled-momentum}. Together with $\xi_{k-1}(s)+\frac{p^k}{s}=\xi_{k-1}(A_k)+\frac{p^k}{A_k}$, which holds since $\xi_{k-1}(s)+s\xi'_{k-1}(s)=\xi_{k-1}(s)+\frac{p^k}{s}$ is constant along the curve, we obtain
\begin{equation}\label{eq:fixed-boundary}
\begin{aligned}
    \xi_k(A_{k+1})= \ &\xi_{k-1}(A_{k+1})-\frac{(A_{k+1}-A_k)(A_{k+1}-u_k)}{LA_{k+1}}\nabla f(x^k),\\
    \xi_k(A_{k+1})+\frac{p^{k+1}}{A_{k+1}}= \ &\xi_{k-1}(A_k)+\frac{p^k}{A_k}-\frac{A_{k+1}-A_k}{L}\nabla f(x^k).
\end{aligned}
\end{equation}
Figure~\ref{fig:correction} illustrates the endpoint corrections at \(A_{k+1}\) induced by the gradient update at \(u_k\).
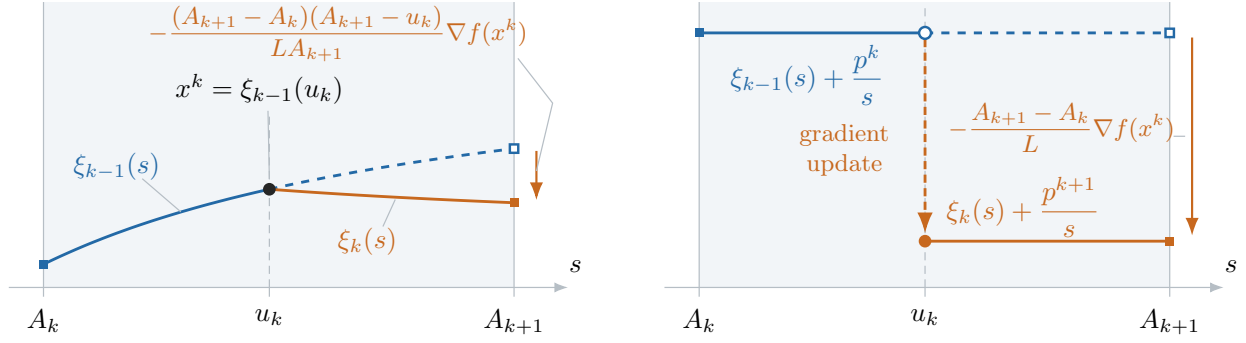
\begin{figure}[t]
    \centering
    \resizebox{\textwidth}{!}{
\begin{tikzpicture}[x=1cm,y=1cm,font=\small,>=Latex]
    \definecolor{hbblue}{RGB}{35,105,166}
    \definecolor{hborange}{RGB}{199,104,30}
    \definecolor{hbgray}{RGB}{143,155,166}
    \definecolor{hblight}{RGB}{242,245,248}
    \tikzset{old/.style={hbblue,line width=1.05pt},
        new/.style={hborange,line width=1.05pt},
        guide/.style={hbgray!65,line width=.45pt},
        dot/.style={circle,fill=black!85,inner sep=1.65pt},
        oldsquare/.style={rectangle,fill=hbblue,inner sep=1.65pt},
        newsquare/.style={rectangle,fill=hborange,inner sep=1.65pt}}
    \path[use as bounding box] (0,-.62) rectangle (16.55,4.27);
    \begin{scope}
        \fill[hblight] (.8,.1) rectangle (6.9,3.8);
        \draw[guide] (.8,.1)--(.8,3.8);
        \draw[guide] (6.9,.1)--(6.9,3.8);
        \draw[guide,densely dashed] (3.728,.1)--(3.728,2.33);
        \draw[guide,->] (.35,.1)--(7.6,.1);
        \foreach \x/\lab in {.8/A_k,3.728/u_k,6.9/A_{k+1}}{
            \draw[guide] (\x,.1)--(\x,-.02);
            \node[below] at (\x,-.05) {$\lab$};
        }
        \node[anchor=west] at (7.48,.37) {$s$};
        \draw[old,domain=1:1.48,samples=50,variable=\t]
            plot ({.8+6.1*(\t-1)},{.4+3*(1-1/\t)});
        \draw[old,dashed,domain=1.48:2,samples=50,variable=\t]
            plot ({.8+6.1*(\t-1)},{.4+3*(1-1/\t)});
        \draw[new,domain=1.48:2,samples=50,variable=\t]
            plot ({.8+6.1*(\t-1)},
                {.4+3*(1-1/1.48)+(1/1.48-1/\t)*(3-2.7*1.48)});
        \node[oldsquare] at (.8,.4) {};
        \node[dot] at (3.728,1.372973) {};
        \node[rectangle,draw=hbblue,fill=white,line width=.9pt,
            inner sep=1.55pt] at (6.9,1.9) {};
        \node[newsquare] at (6.9,1.198) {};
        \node[text=hbblue] at (1.75,1.67) {$\xi_{k-1}(s)$};
        \draw[guide] (2.08,1.51)--(2.58,1.08);
        \node at (3.60,2.68) {$x^k=\xi_{k-1}(u_k)$};
        \draw[guide] (3.71,2.39)--(3.728,1.5);
        \node[text=hborange] at (5.0,.70) {$\xi_k(s)$};
        \draw[guide] (5.2,.89)--(5.36,1.24);
        \draw[hborange,->,line width=.8pt] (7.19,1.87)--(7.19,1.23);
        \node[text=hborange,font=\footnotesize] at (4.62,3.39)
            {$-\dfrac{(A_{k+1}-A_k)(A_{k+1}-u_k)}{LA_{k+1}}\nabla f(x^k)$};
        \draw[guide] (6.95,3.07)--(7.40,2.49)--(7.19,1.61);
    \end{scope}
    \begin{scope}[xshift=8.5cm]
        \fill[hblight] (.8,.1) rectangle (6.9,3.8);
        \draw[guide] (.8,.1)--(.8,3.8);
        \draw[guide] (6.9,.1)--(6.9,3.8);
        \draw[guide,densely dashed] (3.728,.1)--(3.728,3.72);
        \draw[guide,->] (.35,.1)--(7.6,.1);
        \foreach \x/\lab in {.8/A_k,3.728/u_k,6.9/A_{k+1}}{
            \draw[guide] (\x,.1)--(\x,-.02);
            \node[below] at (\x,-.05) {$\lab$};
        }
        \node[anchor=west] at (7.48,.37) {$s$};
        \draw[old] (.8,3.4)--(3.728,3.4);
        \draw[old,dashed] (3.728,3.4)--(6.9,3.4);
        \draw[new] (3.728,.7)--(6.9,.7);
        \draw[new,densely dashed,->] (3.728,3.32)--(3.728,.80);
        \node[oldsquare] at (.8,3.4) {};
        \node[circle,draw=hbblue,fill=white,line width=.9pt,
            inner sep=1.55pt] at (3.728,3.4) {};
        \node[circle,fill=hborange,inner sep=1.65pt] at (3.728,.7) {};
        \node[rectangle,draw=hbblue,fill=white,line width=.9pt,
            inner sep=1.55pt] at (6.9,3.4) {};
        \node[newsquare] at (6.9,.7) {};
        \node[text=hbblue] at (2.20,2.86) {$\xi_{k-1}(s)+\dfrac{p^k}{s}$};
        \node[text=hborange] at (5.00,1.14) {$\xi_k(s)+\dfrac{p^{k+1}}s$};
        \node[align=center,anchor=east,text=hborange,font=\footnotesize]
            at (3.40,1.88) {gradient\\update};
        \draw[hborange,->,line width=.8pt] (7.19,3.34)--(7.19,.77);
        \node[text=hborange,font=\footnotesize] at (5.50,2.17)
            {$-\dfrac{A_{k+1}-A_k}{L}\nabla f(x^k)$};
        \draw[guide] (6.95,2.06)--(7.10,2.06);
    \end{scope}
\end{tikzpicture}}
    \caption{The two endpoint corrections on $[A_k,A_{k+1}]$.
    The blue dashed lines extend the curve $\xi_{k-1}$ beyond $u_k$ without applying the gradient update, and the orange lines represent the quantities after the update.}
    \label{fig:correction}
\end{figure}

\begin{lemma}\label{lem:fixed-gradient-bounds}
For every $s\in[A_k,A_{k+1}]$, we have
\begin{align*}
   \|\nabla f(\xi_{k-1}(s))-\nabla f(\xi_{k-1}(A_{k+1}))\| &\le\frac{L(A_{k+1}-s)}{A_kA_{k+1}}\|p^k\|, \\
   \|\nabla f(\xi_{k-1}(s))\| &\le\sqrt{\frac{2L\Phi_k}{A_k}}+\frac{L(A_{k+1}-A_k)}{A_kA_{k+1}}\|p^k\|.
\end{align*}
\end{lemma}

\begin{proof}
Recall that $\xi_{k-1}(s)+s\xi_{k-1}'(s)$ is constant with respect to $s$ and $\xi_{k-1}'(s)=p^k/s^2$. We thus have $\xi_{k-1}(s)-\xi_{k-1}(t)=(\frac1t-\frac1s)p^k$. Hence, by the Lipschitz continuity of the gradient,
\begin{equation*}
    \|\nabla f(\xi_{k-1}(s))-\nabla f(\xi_{k-1}(t))\|\le L\Big|\frac1s-\frac1t\Big| \cdot \|p^k\|.
\end{equation*}
Taking $t=A_{k+1}$ and using $s\ge A_k$ proves the first inequality. For the second bound, recall that $\frac1{2L}\|\nabla f(x)\|^2\le f(x)-f^\star\le\frac L2\|x-x^\star\|^2$. Since $A_k[f(\xi_{k-1}(A_k))-f^\star]\le\Phi_k$, we have $\|\nabla f(\xi_{k-1}(A_k))\|\le\sqrt{\frac{2L\Phi_k}{A_k}}$. Therefore, we have
\begin{align*}
    \|\nabla f(\xi_{k-1}(s))\| \le \|\nabla f(\xi_{k-1}(A_k))\|+\frac{L(s-A_k)}{A_ks}\|p^k\| \le \sqrt{\frac{2L\Phi_k}{A_k}}+\frac{L(A_{k+1}-A_k)}{A_kA_{k+1}}\|p^k\|.
\end{align*}
The proof is complete.
\end{proof}

\begin{lemma}\label{lem:fixed-energy-defect}
For every $0 \le k < K$, the potential function \eqref{eq:fixed-energy} satisfies
\begin{equation}\label{eq:fixed-defect}
\begin{aligned}
   &\mathbb E[\Phi_{k+1}\mid\mathcal F_k]-\Phi_k \\
   \le \ &\frac{(A_{k+1}-A_k)^2}{L}\mathbb E\Big[\frac{1}{2}\|\nabla f(x^k)\|^2-\frac{A_{k+1}-u_k}{A_{k+1}-A_k}\langle\nabla f(\xi_{k-1}(A_{k+1})),\nabla f(x^k)\rangle\mid\mathcal F_k\Big]\\
    &+\frac{(A_{k+1}-A_k)^2}{2LA_{k+1}}\mathbb E[(A_{k+1}-u_k)^2\|\nabla f(x^k)\|^2\mid\mathcal F_k].
\end{aligned}
\end{equation}
\end{lemma}

\begin{proof}
For the objective function value, due to \eqref{eq:integral}, we have
\begin{align*}
  &A_{k+1}(f(\xi_{k-1}(A_{k+1}))-f^\star)-A_k(f(\xi_{k-1}(A_k))-f^\star) \\
  \le \ &\int_{A_k}^{A_{k+1}} \inner{\nabla f(\xi_{k-1}(s))}{\xi_{k-1}(s)+s\xi'_{k-1}(s)-x^\star} \dd s \\
  = \ &(A_{k+1}-A_k)\mathbb E[\langle\nabla f(x^k),\xi_{k-1}(A_k)+\frac{p^k}{A_k}-x^\star\rangle \mid\mathcal F_k],
\end{align*}
where the second step uses that $u_k$ is uniformly distributed on $[A_k,A_{k+1}]$ conditional on $\mathcal F_k$, that $x^k=\xi_{k-1}(u_k)$, and that $\xi_{k-1}(s)+s\xi'_{k-1}(s)$ is constant along curves. From the second identity in \eqref{eq:fixed-boundary}, we have $\xi_k(A_{k+1})+\frac{p^{k+1}}{A_{k+1}} = \xi_{k-1}(A_k)+\frac{p^k}{A_k} - \frac{A_{k+1}-A_k}{L}\nabla f(x^k)$, which implies
\begin{equation*}
\begin{aligned}
    &\frac{L}{2}\|\xi_k(A_{k+1})+\frac{p^{k+1}}{A_{k+1}}-x^\star\|^2 - \frac{L}{2}\|\xi_{k-1}(A_k)+\frac{p^k}{A_k}-x^\star\|^2 \\
    = \ & -(A_{k+1}-A_k)\langle\nabla f(x^k),\xi_{k-1}(A_k)+\frac{p^k}{A_k}-x^\star\rangle+\frac{(A_{k+1}-A_k)^2}{2L}\|\nabla f(x^k)\|^2.
\end{aligned}
\end{equation*}
By the definition of the potential function \eqref{eq:fixed-energy}, we obtain
\begin{align*}
    \mathbb E[\Phi_{k+1}\mid\mathcal F_k]-\Phi_k
    \le \ &A_{k+1}\mathbb E[f(\xi_k(A_{k+1}))-f(\xi_{k-1}(A_{k+1}))\mid\mathcal F_k]+\frac{(A_{k+1}-A_k)^2}{2L}\mathbb E[\|\nabla f(x^k)\|^2\mid\mathcal F_k].
\end{align*}
Since $f$ is $L$-smooth, it follows that
\begin{align*}
    &f(\xi_k(A_{k+1}))-f(\xi_{k-1}(A_{k+1}))\\
    \le \ &\inner{\nabla f(\xi_{k-1}(A_{k+1}))}{\xi_k(A_{k+1}) - \xi_{k-1}(A_{k+1})} + \frac{L}{2}\|\xi_k(A_{k+1}) - \xi_{k-1}(A_{k+1})\|^2 \\
    = \ &-\frac{(A_{k+1}-A_k)(A_{k+1}-u_k)}{LA_{k+1}}\inner{\nabla f(\xi_{k-1}(A_{k+1}))}{\nabla f(x^k)} + \frac{(A_{k+1}-A_k)^2(A_{k+1}-u_k)^2}{2LA_{k+1}^2}\|\nabla f(x^k)\|^2,
\end{align*}
where we substitute the first identity in \eqref{eq:fixed-boundary} in the second step. The proof is complete by combining the above inequalities.
\end{proof}

\begin{lemma}\label{lem:fixed-local}
Suppose $A_k\ge1$ and $\theta_k=(A_{k+1}-A_k)/\sqrt{A_k}\le3/8$. Then we have
\begin{equation}\label{eq:fixed-local}
    \mathbb E[\Phi_{k+1}\mid\mathcal F_k] \le(1+\theta_k^3)\Phi_k+\frac L5\theta_k^3\frac{\|p^k\|^2}{A_k^2}.
\end{equation}
\end{lemma}

\begin{proof}
We bound the two expectations in Lemma \ref{lem:fixed-energy-defect} separately. For the first expectation, by completing the square with respect to the gradient difference, it follows that
\begin{align*}
    &\frac12\|\nabla f(x^k)\|^2-\frac{A_{k+1}-u_k}{A_{k+1}-A_k}\langle\nabla f(\xi_{k-1}(A_{k+1})),\nabla f(x^k)\rangle+\Big(\frac12-\frac{u_k-A_k}{A_{k+1}-A_k}\Big)\|\nabla f(\xi_{k-1}(A_{k+1}))\|^2\\
    = \ &\frac{u_k-A_k}{A_{k+1}-A_k}\langle\nabla f(\xi_{k-1}(A_{k+1})),\nabla f(x^k)-\nabla f(\xi_{k-1}(A_{k+1}))\rangle+\frac12\|\nabla f(x^k)-\nabla f(\xi_{k-1}(A_{k+1}))\|^2.
\end{align*}
By the conditional uniform-moment identities in Lemma~\ref{lem:technical-interval-calculus}, $\mathbb E[\frac{u_k-A_k}{A_{k+1}-A_k}\mid\mathcal F_k]=\frac12$. Moreover, $\xi_{k-1}(A_{k+1})$ is $\mathcal F_k$-measurable. Consequently, the term $(\frac12-\frac{u_k-A_k}{A_{k+1}-A_k})\|\nabla f(\xi_{k-1}(A_{k+1}))\|^2$ vanishes after taking the conditional expectation, which yields
\begin{align*}
 &\mathbb E\Big[\frac12\|\nabla f(x^k)\|^2
  -\frac{A_{k+1}-u_k}{A_{k+1}-A_k}
       \langle\nabla f(\xi_{k-1}(A_{k+1})),\nabla f(x^k)\rangle
       \mid\mathcal F_k\Big]\\
 ={}&\mathbb E\Big[\frac{u_k-A_k}{A_{k+1}-A_k}
       \langle\nabla f(\xi_{k-1}(A_{k+1})),
          \nabla f(x^k)-\nabla f(\xi_{k-1}(A_{k+1}))\rangle
       \mid\mathcal F_k\Big]\\
 &+\frac12\mathbb E\big[\|\nabla f(x^k)-\nabla f(\xi_{k-1}(A_{k+1}))\|^2
       \mid\mathcal F_k\big]\\
 \le{}&\mathbb E\Big[\frac{u_k-A_k}{A_{k+1}-A_k}
       \|\nabla f(\xi_{k-1}(A_{k+1}))\|
       \|\nabla f(x^k)-\nabla f(\xi_{k-1}(A_{k+1}))\|
       \mid\mathcal F_k\Big]\\
 &+\frac12\mathbb E\big[\|\nabla f(x^k)-\nabla f(\xi_{k-1}(A_{k+1}))\|^2
       \mid\mathcal F_k\big].
\end{align*}
Applying the first inequality in Lemma \ref{lem:fixed-gradient-bounds} gives
\begin{align*}
    &\mathbb E\Big[\frac12\|\nabla f(x^k)\|^2-\frac{A_{k+1}-u_k}{A_{k+1}-A_k}\langle\nabla f(\xi_{k-1}(A_{k+1})),\nabla f(x^k)\rangle\mid\mathcal F_k\Big]\\
    \le \ &\mathbb E\Big[\frac{L(A_{k+1}-u_k)(u_k-A_k)}{(A_{k+1}-A_k)A_kA_{k+1}}\|\nabla f(\xi_{k-1}(A_{k+1}))\| \|p^k\|+\frac{L^2(A_{k+1}-u_k)^2}{2A_k^2A_{k+1}^2}\|p^k\|^2\mid\mathcal F_k\Big] \\
    = \ &\frac{L}{(A_{k+1}-A_k)A_kA_{k+1}}\|\nabla f(\xi_{k-1}(A_{k+1}))\|\|p^k\| \cdot \mathbb E[(u_k-A_k)(A_{k+1}-u_k)\mid\mathcal F_k] \\
    &+\frac{L^2}{2A_k^2A_{k+1}^2}\|p^k\|^2 \cdot \mathbb E[(A_{k+1}-u_k)^2\mid\mathcal F_k].
\end{align*}
The conditional uniform-moment identities in Lemma~\ref{lem:technical-interval-calculus} give
\begin{align*}
    \mathbb E[(u_k-A_k)(A_{k+1}-u_k)\mid\mathcal F_k]&=\frac{(A_{k+1}-A_k)^2}{6},\\
    \mathbb E[(A_{k+1}-u_k)^2\mid\mathcal F_k]&=\frac{(A_{k+1}-A_k)^2}{3}.
\end{align*}
Therefore, we have
\begin{align*}
    &\mathbb E\Big[\frac12\|\nabla f(x^k)\|^2-\frac{A_{k+1}-u_k}{A_{k+1}-A_k}\langle\nabla f(\xi_{k-1}(A_{k+1})),\nabla f(x^k)\rangle\mid\mathcal F_k\Big] \\
    \le \ &\frac{L(A_{k+1}-A_k)}{6A_kA_{k+1}}\|\nabla f(\xi_{k-1}(A_{k+1}))\|\|p^k\|+\frac{L^2(A_{k+1}-A_k)^2}{6A_k^2A_{k+1}^2}\|p^k\|^2 \\
    \le \ &\frac{L(A_{k+1}-A_k)}{6A_kA_{k+1}}\sqrt{\frac{2L\Phi_k}{A_k}}\|p^k\|+\frac{L^2(A_{k+1}-A_k)^2}{3A_k^2A_{k+1}^2}\|p^k\|^2,
\end{align*}
where the last step uses the second inequality in Lemma \ref{lem:fixed-gradient-bounds}.

It remains to bound the second expectation in Lemma \ref{lem:fixed-energy-defect}. Since $x^k=\xi_{k-1}(u_k)$, by applying the second inequality in Lemma \ref{lem:fixed-gradient-bounds} at $s=u_k$ and $(a+b)^2\le2a^2+2b^2$, we obtain
\begin{equation*}
    \|\nabla f(x^k)\|^2\le\frac{4L\Phi_k}{A_k}+\frac{2L^2(A_{k+1}-A_k)^2}{A_k^2A_{k+1}^2}\|p^k\|^2,
\end{equation*}
which further implies
\begin{equation*}
    \frac{(A_{k+1}-A_k)^2}{2LA_{k+1}}\mathbb E[(A_{k+1}-u_k)^2\|\nabla f(x^k)\|^2\mid\mathcal F_k]\le\frac{2(A_{k+1}-A_k)^4}{3A_kA_{k+1}}\Phi_k+\frac{L(A_{k+1}-A_k)^6}{3A_k^2A_{k+1}^3}\|p^k\|^2.
\end{equation*}
We now substitute the above two upper bounds into \eqref{eq:fixed-defect} and get
\begin{equation*}
\begin{aligned}
   &\mathbb E[\Phi_{k+1}\mid\mathcal F_k]-\Phi_k \\
    \le \ &\frac{(A_{k+1}-A_k)^3}{6A_kA_{k+1}}\sqrt{\frac{2L\Phi_k}{A_k}}\|p^k\|+\frac{L(A_{k+1}-A_k)^4}{3A_k^2A_{k+1}^2}\|p^k\|^2+\frac{2(A_{k+1}-A_k)^4}{3A_kA_{k+1}}\Phi_k+\frac{L(A_{k+1}-A_k)^6}{3A_k^2A_{k+1}^3}\|p^k\|^2. 
\end{aligned}
\end{equation*}
Recall that $\theta_k=(A_{k+1}-A_k)/\sqrt{A_k}$. Since $A_{k+1}\ge A_k$, the first term above satisfies
\begin{equation*}
    \frac{(A_{k+1}-A_k)^3}{6A_kA_{k+1}}\sqrt{\frac{2L\Phi_k}{A_k}}\|p^k\|=\theta_k^3\frac{\sqrt{2L\Phi_k}}{6A_{k+1}}\|p^k\|\le\theta_k^3\frac{\sqrt{2L\Phi_k}}{6}\frac{\|p^k\|}{A_k} \le \theta_k^3\Big(\frac{\Phi_k}{2}+\frac{L}{36}\frac{\|p^k\|^2}{A_k^2}\Big).
\end{equation*}
For the rest three terms, we have
\begin{align*}
    \frac{L(A_{k+1}-A_k)^4}{3A_k^2A_{k+1}^2}\|p^k\|^2\le\frac{L}{3}\theta_k^4\frac{\|p^k\|^2}{A_k^2}, \ \frac{2(A_{k+1}-A_k)^4}{3A_kA_{k+1}}\Phi_k\le\frac{2}{3}\theta_k^4\Phi_k, \ \frac{L(A_{k+1}-A_k)^6}{3A_k^2A_{k+1}^3}\|p^k\|^2\le\frac{L}{3}\theta_k^6\frac{\|p^k\|^2}{A_k^2}.
\end{align*}
Therefore, we conclude
\begin{equation*}
    \mathbb E[\Phi_{k+1}\mid\mathcal F_k]-\Phi_k\le\theta_k^3\Big[\Big(\frac12+\frac23\theta_k\Big)\Phi_k+L\Big(\frac1{36}+\frac13\theta_k+\frac13\theta_k^3\Big)\frac{\|p^k\|^2}{A_k^2}\Big].
\end{equation*}
Finally, since $\theta_k\le3/8$, we have $\frac12+\frac23\theta_k\le1$ and $\frac1{36}+\frac13\theta_k+\frac13\theta_k^3<\frac15$. We obatin
\begin{equation*}
    \mathbb E[\Phi_{k+1}\mid\mathcal F_k]-\Phi_k\le\theta_k^3\Phi_k+\frac{L}{5}\theta_k^3\frac{\|p^k\|^2}{A_k^2},
\end{equation*}
and the proof is complete.
\end{proof}

To control the rescaled momentum term in \eqref{eq:fixed-local}, we now bound the optimality distance from $\xi_{k-1}(A_k)$ to $x^\star$.

\begin{lemma}\label{lem:fixed-position-recursion}
For every $0\le k<K$, we have
\begin{equation*}
    \mathbb E[\|\xi_k(A_{k+1})-x^\star\|^2\mid\mathcal F_k]\le\frac{A_k}{A_{k+1}}\|\xi_{k-1}(A_k)-x^\star\|^2+\frac{A_{k+1}-A_k}{LA_{k+1}}\Phi_k+\frac{2(A_{k+1}-A_k)}{LA_{k+1}}\mathbb E[\Phi_{k+1}\mid\mathcal F_k].
\end{equation*}
\end{lemma}

\begin{proof}
The identities in \eqref{eq:fixed-boundary} imply
\begin{equation*}
    \xi_k(A_{k+1})=\frac{A_k}{A_{k+1}}\xi_{k-1}(A_k)+\frac{u_k-A_k}{A_{k+1}}\Big(\xi_{k-1}(A_k)+\frac{p^k}{A_k}\Big)+\frac{A_{k+1}-u_k}{A_{k+1}}\Big(\xi_k(A_{k+1})+\frac{p^{k+1}}{A_{k+1}}\Big).
\end{equation*}
The coefficients are nonnegative and sum to one. By the convexity of the squared norm and the definition of the potential, we have
\begin{equation*}
    \|\xi_k(A_{k+1})-x^\star\|^2\le\frac{A_k}{A_{k+1}}\|\xi_{k-1}(A_k)-x^\star\|^2+\frac{2(u_k-A_k)}{LA_{k+1}}\Phi_k+\frac{2(A_{k+1}-u_k)}{LA_{k+1}}\Phi_{k+1}.
\end{equation*}
Since $\Phi_k$ is $\mathcal F_k$-measurable and $u_k$ is conditionally uniform on $[A_k,A_{k+1}]$, we have $\mathbb E[(u_k-A_k)\Phi_k\mid\mathcal F_k]=\frac{A_{k+1}-A_k}{2}\Phi_k$. By using the nonnegativity of $\Phi_{k+1}$ and $A_{k+1}-u_k\le A_{k+1}-A_k$, we obtain $\mathbb E[(A_{k+1}-u_k)\Phi_{k+1}\mid\mathcal F_k]\le(A_{k+1}-A_k)\mathbb E[\Phi_{k+1}\mid\mathcal F_k]$. Taking conditional expectations completes the proof.
\end{proof}

\subsection{Convergence guarantees} We now establish the convergence rate in terms of the objective function value gap by combining the potential decrease in Lemma~\ref{lem:fixed-local} with the distance recursion in Lemma~\ref{lem:fixed-position-recursion}.
\begin{lemma}\label{lem:fixed-global}
Suppose $\theta_k\le3/8$ for every $0\le k<K$. For every $0\le k\le K$, we have
\begin{equation*}
    \mathbb E[\Phi_k]\le L\|x^0-x^\star\|^2\exp\Big(3\sum_{\ell=0}^{k-1}\theta_\ell^3\Big), \ \text{and} \ \mathbb E[\|\xi_{k-1}(A_k)-x^\star\|^2]\le3\|x^0-x^\star\|^2\exp\Big(3\sum_{\ell=0}^{k-1}\theta_\ell^3\Big).
\end{equation*}
\end{lemma}
\begin{proof}
We prove the results by induction on $k$. Since $A_0=1$, $\xi_{-1}(A_0)=x^0$, and $p^0=0$, we have $\Phi_0=f(x^0)-f^\star+\frac L2\|x^0-x^\star\|^2\le L\|x^0-x^\star\|^2$. Also, $\|\xi_{-1}(A_0)-x^\star\|^2\le \|x^0-x^\star\|^2\le3\|x^0-x^\star\|^2$. Thus, the base case holds at $k=0$. Now suppose both bounds hold at some $0\le k<K$. Note that $\frac{p^k}{A_k} = (\xi_{k-1}(A_k)+\frac{p^k}{A_k}-x^\star) - (\xi_{k-1}(A_k)-x^\star)$. We then have
\begin{equation*}
\begin{aligned}
    \mathbb E\Big[\frac{\|p^k\|^2}{A_k^2}\Big]
    \le \ &2\mathbb E[\|\xi_{k-1}(A_k)+\frac{p^k}{A_k}-x^\star\|^2]+2\mathbb E[\|\xi_{k-1}(A_k)-x^\star\|^2]\\
    \le \ &\frac4L\mathbb E[\Phi_k]+2\mathbb E[\|\xi_{k-1}(A_k)-x^\star\|^2] \\
    \le\ &10\|x^0-x^\star\|^2\exp\Big(3\sum_{\ell=0}^{k-1}\theta_\ell^3\Big). 
\end{aligned}
\end{equation*}
By taking expectations in Lemma~\ref{lem:fixed-local} and using the induction hypothesis, we obtain
\begin{align*}
    \mathbb E[\Phi_{k+1}]\le (1+\theta_k^3)\mathbb E[\Phi_k]+\frac L5\theta_k^3\mathbb E\Big[\frac{\|p^k\|^2}{A_k^2}\Big] \le L\|x^0-x^\star\|^2(1+3\theta_k^3)\exp\Big(3\sum_{\ell=0}^{k-1}\theta_\ell^3\Big).
\end{align*}
Since $1+3\theta_k^3\le \ee^{3\theta_k^3}$, it follows that $\mathbb E[\Phi_{k+1}]\le L\|x^0-x^\star\|^2\exp\Big(3\sum_{\ell=0}^{k}\theta_\ell^3\Big)$, which proves the potential bound at iteration $k+1$. For the optimality distance at that iteration, taking expectations in Lemma~\ref{lem:fixed-position-recursion} gives
\begin{align*}
    \mathbb E[\|\xi_k(A_{k+1})-x^\star\|^2]
    \le \ &\frac{A_k}{A_{k+1}}\mathbb E[\|\xi_{k-1}(A_k)-x^\star\|^2]+\frac{A_{k+1}-A_k}{LA_{k+1}}\mathbb E[\Phi_k]+\frac{2(A_{k+1}-A_k)}{LA_{k+1}}\mathbb E[\Phi_{k+1}]\\
    \le \ &\frac{3A_k+3(A_{k+1}-A_k)}{A_{k+1}}\|x^0-x^\star\|^2\exp\Big(3\sum_{\ell=0}^{k}\theta_\ell^3\Big) \\
    = \ &3\|x^0-x^\star\|^2\exp\Big(3\sum_{\ell=0}^{k}\theta_\ell^3\Big),
\end{align*}
where we have used the induction hypothesis at iteration $k$, $\mathbb E[\Phi_k]\le L\|x^0-x^\star\|^2\exp\Big(3\sum_{\ell=0}^{k-1}\theta_\ell^3\Big)$, and the nonnegativity of $\theta_k^3$. The proof is complete.
\end{proof}

Note that $x^K=\xi_{K-1}(A_K)$ since $u_K=A_K$. The definition of the potential function \eqref{eq:fixed-energy} yields $A_K[f(x^K)-f^\star]\le\Phi_K$. We thus have
\begin{equation}\label{eq:fixed-general-weights}
    \mathbb E[f(x^K)-f^\star]\le\frac{\mathbb E[\Phi_K]}{A_K}\le\frac{L\|x^0-x^\star\|^2}{A_K}\exp\Big(3\sum_{k=0}^{K-1}\theta_k^3\Big).
\end{equation}
The bound \eqref{eq:fixed-general-weights} suggests making $A_K$ large while keeping $\sum_{k=0}^{K-1}\theta_k^3$ bounded. To relate these two requirements, note that
\begin{equation*}
    \theta_k=\frac{\sqrt{A_{k+1}}+\sqrt{A_k}}{\sqrt{A_k}}\big(\sqrt{A_{k+1}}-\sqrt{A_k}\big)\ge2\big(\sqrt{A_{k+1}}-\sqrt{A_k}\big).
\end{equation*}
By the convexity of the cubic function, it follows that
\begin{equation*}
    \sum_{k=0}^{K-1}\theta_k^3\ge8\sum_{k=0}^{K-1}\big(\sqrt{A_{k+1}}-\sqrt{A_k}\big)^3\ge\frac8{K^2}\Big(\sum_{k=0}^{K-1}\big(\sqrt{A_{k+1}}-\sqrt{A_k}\big)\Big)^3=\frac8{K^2}\big(\sqrt{A_K}-1\big)^3.
\end{equation*}
Therefore, a uniform bound on $\sum_{k=0}^{K-1}\theta_k^3$ necessarily requires $(\sqrt{A_K}-1)^3/K^2 = \cO(1)$, and hence $A_K = \cO(K^{4/3})$. This is a restriction of the sufficient error-budget condition in \eqref{eq:fixed-general-weights}, not a lower bound for every predefined schedule or every deterministic choice of time boundaries.

\begin{theorem}\label{thm:fixed}
Let $f$ be convex and $L$-smooth, and let $x^\star$ be a minimizer. Fix the total iteration count $K\ge1$, and choose $A_k=\big(1+\frac{k}{6K^{1/3}}\big)^2$. Independently sample $u_k\sim\operatorname{Unif}(A_k,A_{k+1})$ for $0\le k<K$, set $u_K=A_K$, and use the schedule \eqref{eq:fixed-parameters}. Then the heavy-ball method satisfies
\begin{equation*}
    \mathbb E[f(x^K)-f^\star]\le\frac{36\ee^{1/7}L\|x^0-x^\star\|^2}{K^{4/3}}.
\end{equation*}
\end{theorem}

\begin{proof}
The fixed-time grid properties collected in Lemma~\ref{lem:technical-fixed-grid} give
$\theta_k<3/8$ for every $0\le k<K$, $3\sum_{k=0}^{K-1}\theta_k^3<1/7$, and
$A_K\ge K^{4/3}/36$. Hence the condition of Lemma~\ref{lem:fixed-local} is satisfied, and \eqref{eq:fixed-general-weights} gives
\begin{align*}
    \mathbb E[f(x^K)-f^\star] \le \frac{L\|x^0-x^\star\|^2}{A_K}\exp\Big(3\sum_{k=0}^{K-1}\theta_k^3\Big) \le \frac{36\ee^{1/7}L\|x^0-x^\star\|^2}{K^{4/3}}.
\end{align*}
which completes the proof.
\end{proof}

\section{Deterministic Time Boundaries: Anytime Convergence}\label{sec:anytime-proof}
We now retain deterministic boundaries but require a single infinite schedule. The parameterization and potential are strengthened to control the error from each boundary to the sampled iterate. The construction in this section proves the anytime $\cO(1/K^{4/3})$ guarantee. The same curve, coefficient formula, and potential function will also be used in Section~\ref{sec:random-boundaries}, where the boundaries themselves are randomized.

\subsection{Schedule and interpolation} Fix an iteration $k$ and consider the update after the gradient at $x^k$ is evaluated and the rescaled momentum is updated to $p^{k+1}$. As in the fixed-time construction, we represent the iterate update by an auxiliary twice-differentiable curve $\xi_k(s)$ starting from $x^k$ at time $u_k>0$ and moving along the direction $p^{k+1}$. Instead of requiring $\xi_k(s)+s\xi_k'(s)$ to be constant, we choose the parameterization so that
\begin{equation*}
    \frac{\dd}{\dd s}\bigl(\xi_k(s)+s\xi_k'(s)\bigr)=-\xi_k'(s).
\end{equation*}
Under this choice, the sum of the two squared distances satisfies
\begin{equation}\label{eq:two-squared-distances}
\begin{aligned}
    &\frac{\dd}{\dd s}\frac{L}{2}\bigl[\|\xi_k(s)-x^\star\|^2+\|\xi_k(s)+s\xi_k'(s)-x^\star\|^2\bigr] \\
    = \ &L\langle\xi_k(s)-x^\star,\xi_k'(s)\rangle-L\langle\xi_k(s)+s\xi_k'(s)-x^\star,\xi_k'(s)\rangle \\
    = \ &-Ls\|\xi_k'(s)\|^2.
\end{aligned}
\end{equation}
Since the curve moves along the fixed direction $p^{k+1}$, there exists a scalar function $a_k(s)$ such that $\xi_k'(s)=a_k(s)p^{k+1}$. Consequently, the required identity becomes
\begin{align*}
    3\xi_k'(s)+s\xi_k''(s)=\bigl(3a_k(s)+sa_k'(s)\bigr)p^{k+1}=0.
\end{align*}
It therefore suffices to choose $sa_k'(s)+3a_k(s)=0$. We normalize the direction by choosing the solution $a_k(s)=s^{-3}$. Integrating $\xi_k'(s)=p^{k+1}/s^3$ from $u_k$ yields
\begin{equation}\label{eq:anytime-curve}
    \xi_k(s)=x^k+\frac12\Big(\frac1{u_k^2}-\frac1{s^2}\Big)p^{k+1}, \ \forall s\ge u_k.
\end{equation}

We choose where to evaluate the next gradient along the curve $\xi_k(\cdot)$. Integrating \eqref{eq:integral} over $[A_{k+1},A_{k+2}]$ gives
\begin{align*}
    A_{k+2}[f(\xi_k(A_{k+2}))-f^\star]-A_{k+1}[f(\xi_k(A_{k+1}))-f^\star] \le \int_{A_{k+1}}^{A_{k+2}}\langle\nabla f(\xi_k(s)),\xi_k(s)+\frac{p^{k+1}}{s^2}-x^\star\rangle\dd s.
\end{align*}
By sampling $u_{k+1}\sim\operatorname{Unif}(A_{k+1},A_{k+2})$ independently of the previously sampled times, we update
\begin{align*}
    x^{k+1}=\xi_k(u_{k+1})=x^k+\frac12\Big(\frac1{u_k^2}-\frac1{u_{k+1}^2}\Big)p^{k+1}.
\end{align*}
Comparing the above identity with \eqref{eq:rescaled-two-step} gives $\delta_k=\frac12(\frac1{u_k^2}-\frac1{u_{k+1}^2})$. We also obtain
\begin{align*}
    \int_{A_{k+1}}^{A_{k+2}}\langle\nabla f(\xi_k(s)),\xi_k(s)+\frac{p^{k+1}}{s^2}-x^\star\rangle\dd s = (A_{k+2}-A_{k+1})\mathbb E[\langle\nabla f(x^{k+1}),x^{k+1}+\frac{p^{k+1}}{u_{k+1}^2}-x^\star\rangle\mid \cF_{k+1}].
\end{align*}
We now determine the next stepsize. At $x^{k+1}$, the rescaled momentum update implies
\begin{equation*}
    x^{k+1}+\frac{p^{k+2}}{u_{k+1}^2}-x^\star
    =x^{k+1}+\frac{p^{k+1}}{u_{k+1}^2}-x^\star-\frac{\eta_{k+1}}{u_{k+1}^2\delta_{k+1}}\nabla f(x^{k+1}).
\end{equation*}
Squaring both sides and rearranging yields
\begin{align*}
    &\langle\nabla f(x^{k+1}),x^{k+1}+\frac{p^{k+1}}{u_{k+1}^2}-x^\star\rangle
    -\frac{\eta_{k+1}}{2u_{k+1}^2\delta_{k+1}}\|\nabla f(x^{k+1})\|^2 \\
    = \ &\frac{u_{k+1}^2\delta_{k+1}}{2\eta_{k+1}}\Big[\|x^{k+1}+\frac{p^{k+1}}{u_{k+1}^2}-x^\star\|^2
    -\|x^{k+1}+\frac{p^{k+2}}{u_{k+1}^2}-x^\star\|^2\Big].
\end{align*}
To cancel the sampled inner product with the change in the squared distance, we set
\begin{align*}
    \frac{(A_{k+2}-A_{k+1})u_{k+1}^2\delta_{k+1}}{2\eta_{k+1}}=\frac{L}{2},
\end{align*}
which yields the stepsize selection. 

Under the above development, for every $k\ge0$, independently sample $u_k\sim\operatorname{Unif}(A_k,A_{k+1})$ and use the schedule
\begin{equation}\label{eq:anytime-parameters}
    \delta_k=\frac12\Big(\frac1{u_k^2}-\frac1{u_{k+1}^2}\Big)>0, \ \eta_k=\frac{(A_{k+1}-A_k)u_k^2\delta_k}{L}, \ \beta_k=\frac{\delta_k}{\delta_{k-1}}\ (k\ge1).
\end{equation}
All $\delta_k$ are positive almost surely. These parameters give exactly the coefficients in \eqref{eq:convex-coefficients}. The rescaled momentum update simplifies to
\begin{equation}\label{eq:anytime-rescaled-momentum}
    p^{k+1}=p^k-\frac{(A_{k+1}-A_k)u_k^2}{L}\nabla f(x^k).
\end{equation}
Define $\xi_{-1}(s)=x^0$ for $s\ge A_0$. Then $x^k=\xi_{k-1}(u_k)$ for every $k\ge0$. 

\subsection{Potential function decrease} The fixed-time proof controls the accumulated increase in the potential over a fixed number of iterations. For an anytime guarantee, we instead seek a decrease on each interval that controls the remaining rescaled momentum term in \eqref{eq:fixed-local}. Therefore, we include the squared distance of the iterate itself in the potential. We introduce the potential function over the interval $[A_k,A_{k+1}]$ as follows
\begin{equation}\label{eq:anytime-energy}
    \Psi_k=A_k[f(\xi_{k-1}(A_k))-f^\star]+\frac{L}{2}\|\xi_{k-1}(A_k)-x^\star\|^2 +\frac{L}{2}\|\xi_{k-1}(A_k)+\frac{p^k}{A_k^2}-x^\star\|^2.
\end{equation}
For $k\ge1$, recall that
\begin{align*}
    \xi_{k-1}(s)=\ &x^{k-1}+\frac12\Big(\frac1{u_{k-1}^2}-\frac1{s^2}\Big)p^k,\\
    x^k=\xi_{k-1}(u_k)=\ &x^{k-1}+\frac12\Big(\frac1{u_{k-1}^2}-\frac1{u_k^2}\Big)p^k.
\end{align*}
Subtracting the second identity from the first gives $\xi_{k-1}(s)=x^k+\frac12(\frac1{u_k^2}-\frac1{s^2})p^k$. This identity also holds at $k=0$ since $p^0=0$. Together with \eqref{eq:anytime-curve} and \eqref{eq:anytime-rescaled-momentum}, we obtain
\begin{align*}
    \xi_k(s)-\xi_{k-1}(s)
    =\frac12\Big(\frac1{u_k^2}-\frac1{s^2}\Big)(p^{k+1}-p^k)
    =-\frac{A_{k+1}-A_k}{2L}\Big(1-\Big(\frac{u_k}{s}\Big)^2\Big)\nabla f(x^k), \ \forall s\ge u_k.
\end{align*}
Consequently, the two endpoint corrections satisfy
\begin{equation}\label{eq:anytime-boundary}
\begin{aligned}
    \xi_k(A_{k+1})=\ &\xi_{k-1}(A_{k+1})-\frac{A_{k+1}-A_k}{2L}\Big(1-\Big(\frac{u_k}{A_{k+1}}\Big)^2\Big)\nabla f(x^k),\\
    \xi_k(A_{k+1})+\frac{p^{k+1}}{A_{k+1}^2}=\ &\xi_{k-1}(A_{k+1})+\frac{p^k}{A_{k+1}^2}-\frac{A_{k+1}-A_k}{2L}\Big(1+\Big(\frac{u_k}{A_{k+1}}\Big)^2\Big)\nabla f(x^k).
\end{aligned}
\end{equation}
Similarly, Figure~\ref{fig:anytime-correction} illustrates the endpoint corrections at \(A_{k+1}\) induced by the gradient update at \(u_k\).

\begin{figure}[t]
    \centering
    \resizebox{\textwidth}{!}{
\begin{tikzpicture}[x=1cm,y=1cm,font=\small,>=Latex]
    \definecolor{hbblue}{RGB}{35,105,166}
    \definecolor{hborange}{RGB}{199,104,30}
    \definecolor{hbgray}{RGB}{143,155,166}
    \definecolor{hblight}{RGB}{242,245,248}
    \tikzset{
        old/.style={hbblue,line width=1.05pt},
        new/.style={hborange,line width=1.05pt},
        guide/.style={hbgray!65,line width=.45pt},
        dot/.style={circle,fill=black!85,inner sep=1.65pt},
        oldsquare/.style={rectangle,fill=hbblue,inner sep=1.65pt},
        newsquare/.style={rectangle,fill=hborange,inner sep=1.65pt}
    }

    \path[use as bounding box] (0,-.62) rectangle (17.1,4.95);

    \begin{scope}
        \fill[hblight] (.8,.1) rectangle (6.9,3.9);
        \draw[guide] (.8,.1)--(.8,3.9);
        \draw[guide] (6.9,.1)--(6.9,3.9);
        \draw[guide,densely dashed] (3.728,.1)--(3.728,2.72);
        \draw[guide,->] (.35,.1)--(7.55,.1);

        \foreach \x/\lab in {.8/A_k,3.728/u_k,6.9/A_{k+1}}{
            \draw[guide] (\x,.1)--(\x,-.02);
            \node[below] at (\x,-.05) {$\lab$};
        }
        \node[anchor=west] at (7.43,.37) {$s$};

        \draw[old,domain=1:1.48,samples=60,variable=\t]
            plot ({.8+6.1*(\t-1)},{.45+2.30*(1-1/(\t*\t))/2});
        \draw[old,dashed,domain=1.48:2,samples=60,variable=\t]
            plot ({.8+6.1*(\t-1)},{.45+2.30*(1-1/(\t*\t))/2});
        \draw[new,domain=1.48:2,samples=60,variable=\t]
            plot ({.8+6.1*(\t-1)},
                 {.45+2.30*(1-1/(1.48*1.48))/2
                  +1.10*(1/(1.48*1.48)-1/(\t*\t))/2});

        \node[oldsquare] at (.8,.45) {};
        \node[dot] at (3.728,{.45+2.30*(1-1/(1.48*1.48))/2}) {};
        \node[rectangle,draw=hbblue,fill=white,line width=.9pt,inner sep=1.55pt]
            at (6.9,{.45+2.30*(1-1/4)/2}) {};
        \node[newsquare]
            at (6.9,{.45+2.30*(1-1/(1.48*1.48))/2 + 1.10*(1/(1.48*1.48)-1/4)/2}) {};

        \node[text=hbblue,fill=hblight,inner sep=1.2pt] at (1.85,1.92) {$\xi_{k-1}(s)$};
        \draw[guide] (2.10,1.78)--(2.55,1.33);

        \node[text=hborange,fill=hblight,inner sep=1.2pt] at (5.05,.82) {$\xi_k(s)$};
        \draw[guide] (5.16,.90)--(5.40,1.23);

        \node[fill=hblight,inner sep=1.2pt] at (3.78,3.00) {$x^k=\xi_{k-1}(u_k)$};
        \draw[guide] (3.78,2.76)--(3.728,{.45+2.30*(1-1/(1.48*1.48))/2+.05});

        \draw[hborange,->,line width=.8pt]
            (7.16,{.45+2.30*(1-1/4)/2-.02})--
            (7.16,{.45+2.30*(1-1/(1.48*1.48))/2 + 1.10*(1/(1.48*1.48)-1/4)/2+.02});

        \node[text=hborange,align=center] at (3.95,4.42)
            {$-\dfrac{A_{k+1}-A_k}{2L}
            \Bigl(1-\bigl(\dfrac{u_k}{A_{k+1}}\bigr)^2\Bigr)\nabla f(x^k)$};
        \draw[guide] (6.95,3.50)--(7.32,3.02)--(7.16,2.12);
    \end{scope}

    \begin{scope}[xshift=8.75cm]
        \fill[hblight] (.8,.1) rectangle (6.9,3.9);
        \draw[guide] (.8,.1)--(.8,3.9);
        \draw[guide] (6.9,.1)--(6.9,3.9);
        \draw[guide,densely dashed] (3.728,.1)--(3.728,3.45);
        \draw[guide,->] (.35,.1)--(7.55,.1);

        \foreach \x/\lab in {.8/A_k,3.728/u_k,6.9/A_{k+1}}{
            \draw[guide] (\x,.1)--(\x,-.02);
            \node[below] at (\x,-.05) {$\lab$};
        }
        \node[anchor=west] at (7.43,.37) {$s$};

        \draw[old,domain=1:1.48,samples=60,variable=\t]
            plot ({.8+6.1*(\t-1)},{.45+1.15+1.15/(\t*\t)});
        \draw[old,dashed,domain=1.48:2,samples=60,variable=\t]
            plot ({.8+6.1*(\t-1)},{.45+1.15+1.15/(\t*\t)});
        \draw[new,domain=1.48:2,samples=60,variable=\t]
            plot ({.8+6.1*(\t-1)},
                 {.45+2.30*(1-1/(1.48*1.48))/2
                  +0.55/(1.48*1.48)+0.55/(\t*\t)});

        \node[oldsquare] at (.8,{.45+1.15+1.15}) {};
        \node[circle,draw=hbblue,fill=white,line width=.9pt,inner sep=1.55pt]
            at (3.728,{.45+1.15+1.15/(1.48*1.48)}) {};
        \node[circle,fill=hborange,inner sep=1.65pt]
            at (3.728,{.45+2.30*(1-1/(1.48*1.48))/2 + 0.55/(1.48*1.48)+0.55/(1.48*1.48)}) {};
        \node[rectangle,draw=hbblue,fill=white,line width=.9pt,inner sep=1.55pt]
            at (6.9,{.45+1.15+1.15/4}) {};
        \node[newsquare]
            at (6.9,{.45+2.30*(1-1/(1.48*1.48))/2 + 0.55/(1.48*1.48)+0.55/4}) {};

        \node[text=hbblue,align=center,fill=hblight,inner sep=1.2pt] at (1.90,3.06)
            {$\xi_{k-1}(s)+\dfrac{p^k}{s^2}$};
        \draw[guide] (2.55,2.92)--(3.05,2.70);

        \node[text=hborange,align=center,fill=hblight,inner sep=1.2pt] at (5.22,0.92)
            {$\xi_k(s)+\dfrac{p^{k+1}}{s^2}$};
        \draw[guide] (5.24,1.08)--(5.52,1.62);

        \node[text=hborange,align=center,fill=hblight,inner sep=1.2pt] at (2.80,2.25)
            {gradient\\update};
        \draw[guide] (3.12,2.16)--(3.55,2.05);

        \draw[hborange,->,line width=.8pt]
            (7.16,{.45+1.15+1.15/4-.02})--
            (7.16,{.45+2.30*(1-1/(1.48*1.48))/2 + 0.55/(1.48*1.48)+0.55/4+.02});

        \node[text=hborange,align=center] at (4.70,4.42)
            {$-\dfrac{A_{k+1}-A_k}{2L}
            \Bigl(1+\bigl(\dfrac{u_k}{A_{k+1}}\bigr)^2\Bigr)\nabla f(x^k)$};
        \draw[guide] (6.95,3.48)--(7.30,3.02)--(7.16,1.98);
    \end{scope}
\end{tikzpicture}}
    \caption{The two endpoint corrections in the anytime analysis on $[A_k,A_{k+1}]$. The blue dashed curves extend the incoming curve $\xi_{k-1}$ beyond $u_k$ without applying the gradient update, and the orange curves represent the outgoing curve $\xi_k$ after the update.}
    \label{fig:anytime-correction}
\end{figure}
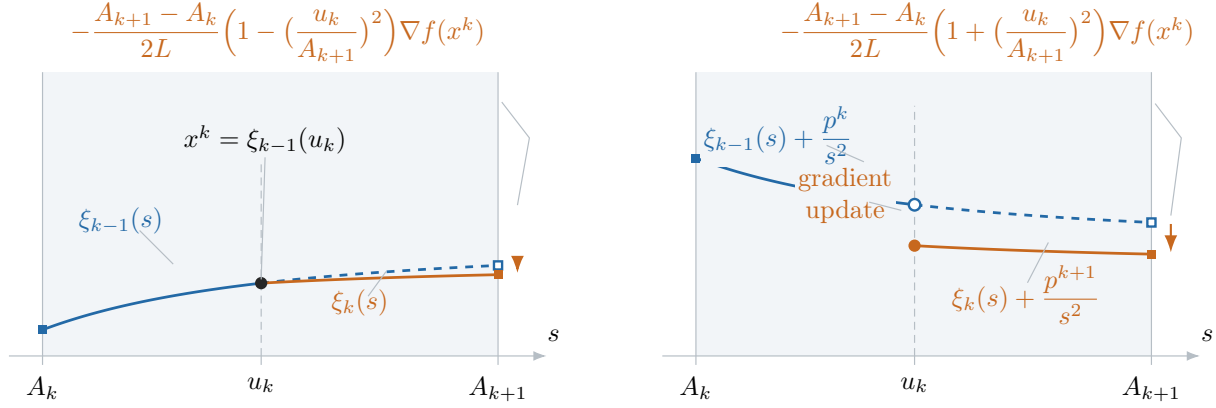

\begin{lemma}\label{lem:anytime-gradient-bounds}
For every $s\in[A_k,A_{k+1}]$, we have
\begin{align*}
    \|\nabla f(\xi_{k-1}(s))-\nabla f(\xi_{k-1}(A_{k+1}))\|
    &\le\frac{L(A_{k+1}-s)}{A_k^3}\|p^k\|,\\
    \|\nabla f(\xi_{k-1}(A_{k+1}))\|
    &\le\big(\mathbb E[\|\nabla f(x^k)\|^2\mid\mathcal F_k]\big)^{1/2}+\frac{L(A_{k+1}-A_k)}{A_k^3}\|p^k\|.
\end{align*}
\end{lemma}

\begin{proof}
The definition of the curve gives $\xi_{k-1}(s)-\xi_{k-1}(t)=\frac12(\frac1{t^2}-\frac1{s^2})p^k$. Hence, by the Lipschitz continuity of the gradient, for every $s,t\in[A_k,A_{k+1}]$, we have
\begin{equation*}
    \|\nabla f(\xi_{k-1}(s))-\nabla f(\xi_{k-1}(t))\| \le\frac{L}{2}\Big|\frac1{s^2}-\frac1{t^2}\Big| \cdot \|p^k\|
    \le\frac{L|s-t|}{A_k^3}\|p^k\|.
\end{equation*}
Taking $t=A_{k+1}$ proves the first inequality. For the second bound, since $x^k=\xi_{k-1}(u_k)$ and $u_k\in[A_k,A_{k+1}]$, we obtain
\begin{equation*}
    \|\nabla f(\xi_{k-1}(A_{k+1}))\|
    \le\|\nabla f(x^k)\|+\frac{L(A_{k+1}-A_k)}{A_k^3}\|p^k\|.
\end{equation*}
The endpoint gradient and $p^k$ are $\mathcal F_k$-measurable. Taking conditional expectations and applying the Cauchy--Schwarz inequality proves the second bound. The proof is complete.
\end{proof}

\begin{lemma}\label{lem:anytime-energy-defect}
For every $k\ge0$, the potential function \eqref{eq:anytime-energy} satisfies
\begin{equation}\label{eq:anytime-defect}
\begin{aligned}
    &\mathbb E[\Psi_{k+1}\mid\mathcal F_k]-\Psi_k\\
    \le\ &-\frac{A_{k+1}-A_k}{2L}\mathbb E[\|\nabla f(x^k)\|^2\mid\mathcal F_k]
    -\frac{L}{4}\Big(\frac1{A_k^4}-\frac1{A_{k+1}^4}\Big)\|p^k\|^2\\
    &+\frac{(A_{k+1}-A_k)^2}{L}\mathbb E\Big[\frac12\|\nabla f(x^k)\|^2-\frac{A_{k+1}-u_k}{A_{k+1}-A_k}\langle\nabla f(\xi_{k-1}(A_{k+1})),\nabla f(x^k)\rangle\mid\mathcal F_k\Big]\\
    &+\frac{2(A_{k+1}-A_k)^2}{A_k^3}\|p^k\| \cdot \mathbb E[\|\nabla f(x^k)\|\mid\mathcal F_k]+\frac{(A_{k+1}-A_k)^4}{2LA_k}\mathbb E[\|\nabla f(x^k)\|^2\mid\mathcal F_k]\\
    &+\frac{(A_{k+1}-A_k)^3}{LA_k}\|\nabla f(\xi_{k-1}(A_{k+1}))\| \cdot \mathbb E[\|\nabla f(x^k)\|\mid\mathcal F_k].
\end{aligned}
\end{equation}
\end{lemma}

\begin{proof}
Since $f$ is convex and $L$-smooth, the smooth convex interpolation inequality \eqref{eq:cocoercivity-inequality} gives
\begin{align*}
 &\frac{\dd}{\dd s}[s(f(\xi_{k-1}(s))-f^\star)]\\
 ={}&f(\xi_{k-1}(s))-f^\star
       +s\langle\nabla f(\xi_{k-1}(s)),\xi'_{k-1}(s)\rangle\\
 \le{}&\langle\nabla f(\xi_{k-1}(s)),\xi_{k-1}(s)-x^\star\rangle
       +s\langle\nabla f(\xi_{k-1}(s)),\xi'_{k-1}(s)\rangle-\frac1{2L}\|\nabla f(\xi_{k-1}(s))\|^2.
\end{align*}
By substituting $\xi'_{k-1}(s)=\frac{p^k}{s^3}$, we further obtain
\begin{align*}
    \frac{\dd}{\dd s}[s(f(\xi_{k-1}(s))-f^\star)] \le \langle\nabla f(\xi_{k-1}(s)),\xi_{k-1}(s)+\frac{p^k}{s^2}-x^\star \rangle-\frac{1}{2L}\|\nabla f(\xi_{k-1}(s))\|^2.
\end{align*}
Combining this bound with \eqref{eq:two-squared-distances} gives
\begin{equation*}
\begin{aligned}
    &\frac{\dd}{\dd s}\Big[s[f(\xi_{k-1}(s))-f^\star]
    +\frac{L}{2}\|\xi_{k-1}(s)-x^\star\|^2
    +\frac{L}{2}\|\xi_{k-1}(s)+\frac{p^k}{s^2}-x^\star\|^2\Big]\\
    \le\ &\langle\nabla f(\xi_{k-1}(s)),\xi_{k-1}(s)+\frac{p^k}{s^2}-x^\star\rangle
    -\frac1{2L}\|\nabla f(\xi_{k-1}(s))\|^2-\frac{L}{s^5}\|p^k\|^2.
\end{aligned}
\end{equation*}
Integrating over $[A_k,A_{k+1}]$ and using uniform sampling yields
\begin{equation}\label{eq:anytime-potential-decrease-I}
    \begin{aligned}
         &A_{k+1}[f(\xi_{k-1}(A_{k+1}))-f^\star]
    +\frac{L}{2}\|\xi_{k-1}(A_{k+1})-x^\star\|^2+\frac{L}{2}\|\xi_{k-1}(A_{k+1})+\frac{p^k}{A_{k+1}^2}-x^\star\|^2-\Psi_k\\
    \le\ &(A_{k+1}-A_k)\mathbb E[\langle\nabla f(x^k),x^k+\frac{p^k}{u_k^2}-x^\star\rangle\mid\mathcal F_k]\\
    &-\frac{A_{k+1}-A_k}{2L}\mathbb E[\|\nabla f(x^k)\|^2\mid\mathcal F_k]
    -\frac{L}{4}\Big(\frac1{A_k^4}-\frac1{A_{k+1}^4}\Big)\|p^k\|^2,
    \end{aligned}
\end{equation}
where we use $\int_{A_k}^{A_{k+1}}s^{-5}\dd s=\frac14(A_k^{-4}-A_{k+1}^{-4})$ in the last term. We next bound the changes at the endpoint $A_{k+1}$. By using the two endpoint identities in \eqref{eq:anytime-boundary} and $\xi_{k-1}(A_{k+1})=x^k+\frac12(\frac1{u_k^2}-\frac1{A_{k+1}^2})p^k$, we obtain
\begin{align*}
    &\frac{L}{2}\big[\|\xi_k(A_{k+1})-x^\star\|^2-\|\xi_{k-1}(A_{k+1})-x^\star\|^2\big]\\
    &+\frac{L}{2}\Big[\|\xi_k(A_{k+1})+\frac{p^{k+1}}{A_{k+1}^2}-x^\star\|^2
    -\|\xi_{k-1}(A_{k+1})+\frac{p^k}{A_{k+1}^2}-x^\star\|^2\Big]\\
    =\ &-(A_{k+1}-A_k)\langle\nabla f(x^k),x^k+\frac{p^k}{u_k^2}-x^\star\rangle+\frac{A_{k+1}-A_k}{2}\Big(\frac1{u_k^2}-\frac{u_k^2}{A_{k+1}^4}\Big)\langle\nabla f(x^k),p^k\rangle\\
    &+\frac{(A_{k+1}-A_k)^2}{4L}\Big(1+\Big(\frac{u_k}{A_{k+1}}\Big)^4\Big)\|\nabla f(x^k)\|^2.
\end{align*}
Since $A_k\le u_k\le A_{k+1}$, we have
\begin{equation*}
    0\le\frac12\Big(\frac1{u_k^2}-\frac{u_k^2}{A_{k+1}^4}\Big)
    =\frac{(A_{k+1}-u_k)(A_{k+1}+u_k)(A_{k+1}^2+u_k^2)}{2u_k^2A_{k+1}^4}
    \le\frac{2(A_{k+1}-u_k)}{A_k^3}.
\end{equation*}
Note that $u_k/A_{k+1}\le1$ and $A_{k+1}-u_k \le A_{k+1} - A_k$. It follows
\begin{align*}
    &\frac{A_{k+1}-A_k}{2}\Big(\frac1{u_k^2}-\frac{u_k^2}{A_{k+1}^4}\Big)\langle\nabla f(x^k),p^k\rangle+\frac{(A_{k+1}-A_k)^2}{4L}\Big(1+\Big(\frac{u_k}{A_{k+1}}\Big)^4\Big)\|\nabla f(x^k)\|^2 \\
    \le \ &\frac{2(A_{k+1}-A_k)^2}{A_k^3}\|p^k\|\|\nabla f(x^k)\|+\frac{(A_{k+1}-A_k)^2}{2L}\|\nabla f(x^k)\|^2,
\end{align*}
which implies
\begin{equation}\label{eq:anytime-potential-decrease-II}
    \begin{aligned}
       &\frac{L}{2}\big[\|\xi_k(A_{k+1})-x^\star\|^2-\|\xi_{k-1}(A_{k+1})-x^\star\|^2\big]\\
       &+\frac{L}{2}\Big[\|\xi_k(A_{k+1})+\frac{p^{k+1}}{A_{k+1}^2}-x^\star\|^2
    -\|\xi_{k-1}(A_{k+1})+\frac{p^k}{A_{k+1}^2}-x^\star\|^2\Big]\\
    \le \ &-(A_{k+1}-A_k)\langle\nabla f(x^k),x^k+\frac{p^k}{u_k^2}-x^\star\rangle+\frac{2(A_{k+1}-A_k)^2}{A_k^3}\|p^k\|\|\nabla f(x^k)\|\\
    &+\frac{(A_{k+1}-A_k)^2}{2L}\|\nabla f(x^k)\|^2.
    \end{aligned}
\end{equation}
For the objective function value at the endpoint, $L$-smoothness and the first identity in \eqref{eq:anytime-boundary} imply
\begin{align*}
    &A_{k+1}[f(\xi_k(A_{k+1}))-f(\xi_{k-1}(A_{k+1}))]\\
    \le\ &-\frac{A_{k+1}(A_{k+1}-A_k)}{2L}\Big(1-\Big(\frac{u_k}{A_{k+1}}\Big)^2\Big)
    \langle\nabla f(\xi_{k-1}(A_{k+1})),\nabla f(x^k)\rangle\\
    &+\frac{A_{k+1}(A_{k+1}-A_k)^2}{8L}\Big(1-\Big(\frac{u_k}{A_{k+1}}\Big)^2\Big)^2\|\nabla f(x^k)\|^2.
\end{align*}
The coefficient of the inner product satisfies
\begin{equation*}
    \frac{A_{k+1}(A_{k+1}-A_k)}{2L}\Big(1-\Big(\frac{u_k}{A_{k+1}}\Big)^2\Big)
    =\frac{(A_{k+1}-A_k)(A_{k+1}-u_k)}{L}
    -\frac{(A_{k+1}-A_k)(A_{k+1}-u_k)^2}{2LA_{k+1}}.
\end{equation*}
Moreover, since $1-(u_k/A_{k+1})^2\le2(A_{k+1}-u_k)/A_{k+1}$, we obtain
\begin{equation}\label{eq:anytime-potential-decrease-III}
\begin{aligned}
    &A_{k+1}[f(\xi_k(A_{k+1}))-f(\xi_{k-1}(A_{k+1}))]\\
    \le \ &-\frac{(A_{k+1}-A_k)(A_{k+1}-u_k)}{L}\langle\nabla f(\xi_{k-1}(A_{k+1})),\nabla f(x^k)\rangle\\
    &+\frac{(A_{k+1}-A_k)^3}{LA_k}\|\nabla f(\xi_{k-1}(A_{k+1}))\|\|\nabla f(x^k)\|+\frac{(A_{k+1}-A_k)^4}{2LA_k}\|\nabla f(x^k)\|^2,
\end{aligned}
\end{equation}
where we also used the Cauchy--Schwarz inequality, $A_{k+1}-u_k\le A_{k+1}-A_k$, and $A_{k+1}\ge A_k$.

Finally, we combine the above three inequalities \eqref{eq:anytime-potential-decrease-I}, \eqref{eq:anytime-potential-decrease-II}, and \eqref{eq:anytime-potential-decrease-III}. Since $\Psi_k$, $\xi_{k-1}(A_{k+1})$, and $p^k$ are $\mathcal F_k$-measurable, the definition of the potential function \eqref{eq:anytime-energy} gives
\begin{align*}
    &\mathbb E[\Psi_{k+1}\mid\mathcal F_k]-\Psi_k \\
    = \ &A_{k+1}[f(\xi_{k-1}(A_{k+1}))-f^\star]+\frac L2\|\xi_{k-1}(A_{k+1})-x^\star\|^2+\frac L2\|\xi_{k-1}(A_{k+1})+\frac{p^k}{A_{k+1}^2}-x^\star\|^2-\Psi_k\\
    &+\frac L2\mathbb E[\|\xi_k(A_{k+1})-x^\star\|^2-\|\xi_{k-1}(A_{k+1})-x^\star\|^2\mid\mathcal F_k]\\
    &+\frac L2\mathbb E\big[\|\xi_k(A_{k+1})+\frac{p^{k+1}}{A_{k+1}^2}-x^\star\|^2 -\|\xi_{k-1}(A_{k+1})+\frac{p^k}{A_{k+1}^2}-x^\star\|^2\mid\mathcal F_k\big] \\
    &+A_{k+1}\mathbb E[f(\xi_k(A_{k+1}))-f(\xi_{k-1}(A_{k+1}))\mid\mathcal F_k].
\end{align*}
By taking conditional expectations in \eqref{eq:anytime-potential-decrease-II} and \eqref{eq:anytime-potential-decrease-III} and adding \eqref{eq:anytime-potential-decrease-I}, we have
\begin{align*}
    &\mathbb E[\Psi_{k+1}\mid\mathcal F_k]-\Psi_k\\
    \le \ &(A_{k+1}-A_k)\mathbb E\big[\langle\nabla f(x^k),x^k+\frac{p^k}{u_k^2}-x^\star\rangle\mid\mathcal F_k\big]\\
    &-\frac{A_{k+1}-A_k}{2L}\mathbb E[\|\nabla f(x^k)\|^2\mid\mathcal F_k]-\frac L4\Big(\frac1{A_k^4}-\frac1{A_{k+1}^4}\Big)\|p^k\|^2\\
    &-(A_{k+1}-A_k)\mathbb E\big[\langle\nabla f(x^k),x^k+\frac{p^k}{u_k^2}-x^\star\rangle\mid\mathcal F_k\big]\\
    &+\frac{2(A_{k+1}-A_k)^2}{A_k^3}\|p^k\|\mathbb E[\|\nabla f(x^k)\|\mid\mathcal F_k]+\frac{(A_{k+1}-A_k)^2}{2L}\mathbb E[\|\nabla f(x^k)\|^2\mid\mathcal F_k]\\
    &-\frac{A_{k+1}-A_k}{L}\mathbb E[(A_{k+1}-u_k)\langle\nabla f(\xi_{k-1}(A_{k+1})),\nabla f(x^k)\rangle\mid\mathcal F_k]\\
    &+\frac{(A_{k+1}-A_k)^3}{LA_k}\|\nabla f(\xi_{k-1}(A_{k+1}))\|\mathbb E[\|\nabla f(x^k)\|\mid\mathcal F_k]+\frac{(A_{k+1}-A_k)^4}{2LA_k}\mathbb E[\|\nabla f(x^k)\|^2\mid\mathcal F_k]\\
    = \ &-\frac{A_{k+1}-A_k}{2L}\mathbb E[\|\nabla f(x^k)\|^2\mid\mathcal F_k]-\frac L4\Big(\frac1{A_k^4}-\frac1{A_{k+1}^4}\Big)\|p^k\|^2\\
    &+\frac{(A_{k+1}-A_k)^2}{L}\mathbb E\Big[\frac12\|\nabla f(x^k)\|^2-\frac{A_{k+1}-u_k}{A_{k+1}-A_k}\langle\nabla f(\xi_{k-1}(A_{k+1})),\nabla f(x^k)\rangle\mid\mathcal F_k\Big]\\
    &+\frac{2(A_{k+1}-A_k)^2}{A_k^3}\|p^k\|\mathbb E[\|\nabla f(x^k)\|\mid\mathcal F_k]+\frac{(A_{k+1}-A_k)^4}{2LA_k}\mathbb E[\|\nabla f(x^k)\|^2\mid\mathcal F_k]\\
    &+\frac{(A_{k+1}-A_k)^3}{LA_k}\|\nabla f(\xi_{k-1}(A_{k+1}))\|\mathbb E[\|\nabla f(x^k)\|\mid\mathcal F_k],
\end{align*}
which completes the proof.
\end{proof}

\begin{lemma}\label{lem:local}
Suppose $A_k\ge1$ and $\theta_k=(A_{k+1}-A_k)/A_k^{1/4}\le1/8$. Then we have
\begin{equation}\label{eq:anytime-local}
\begin{aligned}
    \mathbb E[\Psi_{k+1}\mid\mathcal F_k]\le \Psi_k
    -\frac{A_{k+1}-A_k}{4L}\mathbb E[\|\nabla f(x^k)\|^2\mid\mathcal F_k]
    -\frac{L(A_{k+1}-A_k)}{4A_k^5}\|p^k\|^2.
\end{aligned}
\end{equation}
\end{lemma}

\begin{proof}
We bound the terms in Lemma~\ref{lem:anytime-energy-defect} separately. We first retain a lower bound on the decrease from the rescaled momentum. Since $A_k\ge1$ and $\theta_k\le1/8$, we have
$A_{k+1}/A_k=1+\theta_k/A_k^{3/4}\le9/8$. By applying Lemma~\ref{lem:technical-interval-calculus} with $m=4$, and using $(9/8)^5<2$, it follows
\begin{equation*}
    \frac14\Big(\frac1{A_k^4}-\frac1{A_{k+1}^4}\Big)
    \ge\frac{A_{k+1}-A_k}{A_{k+1}^5}
    \ge\frac{A_{k+1}-A_k}{2A_k^5}.
\end{equation*}
For the third term on the right-hand side of \eqref{eq:anytime-defect}, completing the square with respect to the gradient difference yields
\begin{align*}
    &\frac12\|\nabla f(x^k)\|^2
    -\frac{A_{k+1}-u_k}{A_{k+1}-A_k}\langle\nabla f(\xi_{k-1}(A_{k+1})),\nabla f(x^k)\rangle+\Big(\frac12-\frac{u_k-A_k}{A_{k+1}-A_k}\Big)\|\nabla f(\xi_{k-1}(A_{k+1}))\|^2\\
    =\ &\frac{u_k-A_k}{A_{k+1}-A_k}\langle\nabla f(\xi_{k-1}(A_{k+1})),\nabla f(x^k)-\nabla f(\xi_{k-1}(A_{k+1}))\rangle+\frac12\|\nabla f(x^k)-\nabla f(\xi_{k-1}(A_{k+1}))\|^2.
\end{align*}
By the conditional uniform-moment identities in Lemma~\ref{lem:technical-interval-calculus}, $\mathbb E[\frac{u_k-A_k}{A_{k+1}-A_k}\mid\mathcal F_k]=\frac12$. Moreover, $\xi_{k-1}(A_{k+1})$ is $\mathcal F_k$-measurable. Consequently, the term $(\frac12-\frac{u_k-A_k}{A_{k+1}-A_k})\|\nabla f(\xi_{k-1}(A_{k+1}))\|^2$ vanishes after taking the conditional expectation. Applying the first inequality in Lemma~\ref{lem:anytime-gradient-bounds} gives
\begin{align*}
    &\mathbb E\Big[\frac12\|\nabla f(x^k)\|^2
    -\frac{A_{k+1}-u_k}{A_{k+1}-A_k}\langle\nabla f(\xi_{k-1}(A_{k+1})),\nabla f(x^k)\rangle\mid\mathcal F_k\Big]\\
    \le\ &\frac{L}{(A_{k+1}-A_k)A_k^3}\|\nabla f(\xi_{k-1}(A_{k+1}))\|\|p^k\|
    \cdot\mathbb E[(u_k-A_k)(A_{k+1}-u_k)\mid\mathcal F_k]\\
    &+\frac{L^2}{2A_k^6}\|p^k\|^2\cdot\mathbb E[(A_{k+1}-u_k)^2\mid\mathcal F_k].
\end{align*}
The same lemma gives
\begin{equation*}
    \mathbb E[(u_k-A_k)(A_{k+1}-u_k)\mid\mathcal F_k]
    =\frac{(A_{k+1}-A_k)^2}{6},\ 
    \mathbb E[(A_{k+1}-u_k)^2\mid\mathcal F_k]
    =\frac{(A_{k+1}-A_k)^2}{3}.
\end{equation*}
Therefore, we obtain
\begin{align*}
    &\mathbb E\Big[\frac12\|\nabla f(x^k)\|^2
    -\frac{A_{k+1}-u_k}{A_{k+1}-A_k}\langle\nabla f(\xi_{k-1}(A_{k+1})),\nabla f(x^k)\rangle\mid\mathcal F_k\Big]\\
    \le\ &\frac{L(A_{k+1}-A_k)}{6A_k^3}\|\nabla f(\xi_{k-1}(A_{k+1}))\|\|p^k\|
    +\frac{L^2(A_{k+1}-A_k)^2}{6A_k^6}\|p^k\|^2\\
    \le\ &\frac{L(A_{k+1}-A_k)}{6A_k^3}\|p^k\|\big(\mathbb E[\|\nabla f(x^k)\|^2\mid\mathcal F_k]\big)^{1/2}
    +\frac{L^2(A_{k+1}-A_k)^2}{3A_k^6}\|p^k\|^2,
\end{align*}
where the last step uses the second inequality in Lemma~\ref{lem:anytime-gradient-bounds}. It remains to bound the last three terms in \eqref{eq:anytime-defect}. For the first of these terms, the Cauchy--Schwarz inequality gives
\begin{equation*}
    \frac{2(A_{k+1}-A_k)^2}{A_k^3}\|p^k\|\mathbb E[\|\nabla f(x^k)\|\mid\mathcal F_k]
    \le\frac{2(A_{k+1}-A_k)^2}{A_k^3}\|p^k\|\big(\mathbb E[\|\nabla f(x^k)\|^2\mid\mathcal F_k]\big)^{1/2}.
\end{equation*}
We leave the second term unchanged. For the third of these terms, applying the second inequality in Lemma~\ref{lem:anytime-gradient-bounds} and the Cauchy--Schwarz inequality yields
\begin{equation*}
    \begin{aligned}
       &\frac{(A_{k+1}-A_k)^3}{LA_k}\|\nabla f(\xi_{k-1}(A_{k+1}))\|\mathbb E[\|\nabla f(x^k)\|\mid\mathcal F_k]\\
    \le\ &\frac{(A_{k+1}-A_k)^3}{LA_k}\mathbb E[\|\nabla f(x^k)\|^2\mid\mathcal F_k]
    +\frac{(A_{k+1}-A_k)^4}{A_k^4}\|p^k\|\big(\mathbb E[\|\nabla f(x^k)\|^2\mid\mathcal F_k]\big)^{1/2}. 
    \end{aligned}
\end{equation*}
By substituting the above bounds into \eqref{eq:anytime-defect}, we obtain
\begin{equation}\label{eq:anytime-potential-final}
    \begin{aligned}
        &\mathbb E[\Psi_{k+1}\mid\mathcal F_k]-\Psi_k\\
    \le\ &-\frac{A_{k+1}-A_k}{2L}\mathbb E[\|\nabla f(x^k)\|^2\mid\mathcal F_k]
    -\frac{L(A_{k+1}-A_k)}{2A_k^5}\|p^k\|^2\\
    &+\Big[\frac{2(A_{k+1}-A_k)^2}{A_k^3}
    +\frac{(A_{k+1}-A_k)^3}{6A_k^3}
    +\frac{(A_{k+1}-A_k)^4}{A_k^4}\Big]\|p^k\|\big(\mathbb E[\|\nabla f(x^k)\|^2\mid\mathcal F_k]\big)^{1/2}\\
    &+\frac{L(A_{k+1}-A_k)^4}{3A_k^6}\|p^k\|^2
    +\Big[\frac{(A_{k+1}-A_k)^3}{LA_k}+\frac{(A_{k+1}-A_k)^4}{2LA_k}\Big]\mathbb E[\|\nabla f(x^k)\|^2\mid\mathcal F_k].
    \end{aligned}
\end{equation}
Young's inequality gives
\begin{equation*}
    \|p^k\|\big(\mathbb E[\|\nabla f(x^k)\|^2\mid\mathcal F_k]\big)^{1/2}\le\frac{A_k^{5/2}}2\Big(\frac1L\mathbb E[\|\nabla f(x^k)\|^2\mid\mathcal F_k]+\frac{L}{A_k^5}\|p^k\|^2\Big),
\end{equation*}
which implies that
\begin{align*}
    &\Big[\frac{2(A_{k+1}-A_k)^2}{A_k^3}+\frac{(A_{k+1}-A_k)^3}{6A_k^3}+\frac{(A_{k+1}-A_k)^4}{A_k^4}\Big]\|p^k\|\big(\mathbb E[\|\nabla f(x^k)\|^2\mid\mathcal F_k]\big)^{1/2}\\
    \le \ &\frac{A_k^{5/2}}2\Big[\frac{2(A_{k+1}-A_k)^2}{A_k^3}+\frac{(A_{k+1}-A_k)^3}{6A_k^3}+\frac{(A_{k+1}-A_k)^4}{A_k^4}\Big]\Big(\frac1L\mathbb E[\|\nabla f(x^k)\|^2\mid\mathcal F_k]+\frac{L}{A_k^5}\|p^k\|^2\Big)\\
    = \ &(A_{k+1}-A_k)\Big[\frac{A_{k+1}-A_k}{A_k^{1/2}}+\frac{(A_{k+1}-A_k)^2}{12A_k^{1/2}}+\frac{(A_{k+1}-A_k)^3}{2A_k^{3/2}}\Big]\Big(\frac1L\mathbb E[\|\nabla f(x^k)\|^2\mid\mathcal F_k]+\frac{L}{A_k^5}\|p^k\|^2\Big).
\end{align*}
Since $A_{k+1}-A_k=\theta_kA_k^{1/4}$ and $A_k\ge1$, the coefficient satisfies
\begin{equation*}
    \frac{A_{k+1}-A_k}{A_k^{1/2}}+\frac{(A_{k+1}-A_k)^2}{12A_k^{1/2}}+\frac{(A_{k+1}-A_k)^3}{2A_k^{3/2}}=\frac{\theta_k}{A_k^{1/4}}+\frac{\theta_k^2}{12}+\frac{\theta_k^3}{2A_k^{3/4}}\le\theta_k+\frac{\theta_k^2}{12}+\frac{\theta_k^3}{2}.
\end{equation*}
Consequently, we obtain
\begin{equation}\label{eq:anytime-mixed-term-bound}
\begin{aligned}
    &\Big[\frac{2(A_{k+1}-A_k)^2}{A_k^3}+\frac{(A_{k+1}-A_k)^3}{6A_k^3}+\frac{(A_{k+1}-A_k)^4}{A_k^4}\Big]\|p^k\|\big(\mathbb E[\|\nabla f(x^k)\|^2\mid\mathcal F_k]\big)^{1/2}\\
    \le \ &(A_{k+1}-A_k)\Big(\theta_k+\frac{\theta_k^2}{12}+\frac{\theta_k^3}{2}\Big)\Big(\frac1L\mathbb E[\|\nabla f(x^k)\|^2\mid\mathcal F_k]+\frac{L}{A_k^5}\|p^k\|^2\Big).
\end{aligned}
\end{equation}
For the term involving $\|p^k\|^2$ on the right-hand side of \eqref{eq:anytime-potential-final}, we have
\begin{equation}\label{eq:anytime-momentum-remainder-bound}
\begin{aligned}
   \frac{L(A_{k+1}-A_k)^4}{3A_k^6}\|p^k\|^2=\ &(A_{k+1}-A_k)\frac{(A_{k+1}-A_k)^3}{3A_k}\frac{L}{A_k^5}\|p^k\|^2 \\
    =\ &(A_{k+1}-A_k)\frac{\theta_k^3}{3A_k^{1/4}}\frac{L}{A_k^5}\|p^k\|^2\\
    \le\ &(A_{k+1}-A_k)\frac{\theta_k^3}{3}\frac{L}{A_k^5}\|p^k\|^2.  
\end{aligned}
\end{equation}
For the last term on the right-hand side of \eqref{eq:anytime-potential-final}, we have
\begin{equation}\label{eq:anytime-gradient-remainder-bound}
\begin{aligned}
    &\Big[\frac{(A_{k+1}-A_k)^3}{LA_k}+\frac{(A_{k+1}-A_k)^4}{2LA_k}\Big]\mathbb E[\|\nabla f(x^k)\|^2\mid\mathcal F_k]\\
    = \ &(A_{k+1}-A_k)\Big[\frac{(A_{k+1}-A_k)^2}{A_k}+\frac{(A_{k+1}-A_k)^3}{2A_k}\Big]\frac1L\mathbb E[\|\nabla f(x^k)\|^2\mid\mathcal F_k]\\
    = \ &(A_{k+1}-A_k)\Big[\frac{\theta_k^2}{A_k^{1/2}}+\frac{\theta_k^3}{2A_k^{1/4}}\Big]\frac1L\mathbb E[\|\nabla f(x^k)\|^2\mid\mathcal F_k]\\
    \le \ &(A_{k+1}-A_k)\Big(\theta_k^2+\frac{\theta_k^3}{2}\Big)\frac1L\mathbb E[\|\nabla f(x^k)\|^2\mid\mathcal F_k].
\end{aligned}
\end{equation}
Substituting \eqref{eq:anytime-mixed-term-bound}, \eqref{eq:anytime-momentum-remainder-bound}, and \eqref{eq:anytime-gradient-remainder-bound} into inequality \eqref{eq:anytime-potential-final} gives
\begin{align*}
    &\mathbb E[\Psi_{k+1}\mid\mathcal F_k]-\Psi_k \\
    \le \ &-\frac{A_{k+1}-A_k}{2}\Big(\frac1L\mathbb E[\|\nabla f(x^k)\|^2\mid\mathcal F_k]+\frac{L}{A_k^5}\|p^k\|^2\Big)\\
    &+(A_{k+1}-A_k)\Big(\theta_k+\frac{\theta_k^2}{12}+\frac{\theta_k^3}{2}\Big)\Big(\frac1L\mathbb E[\|\nabla f(x^k)\|^2\mid\mathcal F_k]+\frac{L}{A_k^5}\|p^k\|^2\Big)\\
    &+(A_{k+1}-A_k)\frac{\theta_k^3}{3}\frac{L}{A_k^5}\|p^k\|^2+(A_{k+1}-A_k)\Big(\theta_k^2+\frac{\theta_k^3}{2}\Big)\frac1L\mathbb E[\|\nabla f(x^k)\|^2\mid\mathcal F_k]\\
    = \ &(A_{k+1}-A_k)\Big[-\frac12+\Big(\theta_k+\frac{\theta_k^2}{12}+\frac{\theta_k^3}{2}\Big)+\Big(\theta_k^2+\frac{\theta_k^3}{2}\Big)\Big]\frac1L\mathbb E[\|\nabla f(x^k)\|^2\mid\mathcal F_k]\\
    &+(A_{k+1}-A_k)\Big[-\frac12+\Big(\theta_k+\frac{\theta_k^2}{12}+\frac{\theta_k^3}{2}\Big)+\frac{\theta_k^3}{3}\Big]\frac{L}{A_k^5}\|p^k\|^2\\
    = \ &(A_{k+1}-A_k)\Big(-\frac12+\theta_k+\frac{13}{12}\theta_k^2+\theta_k^3\Big)\frac1L\mathbb E[\|\nabla f(x^k)\|^2\mid\mathcal F_k]\\
    &+(A_{k+1}-A_k)\Big(-\frac12+\theta_k+\frac1{12}\theta_k^2+\frac56\theta_k^3\Big)\frac{L}{A_k^5}\|p^k\|^2.
\end{align*}
Since $\theta_k\ge0$, we further have
\begin{equation*}
    \mathbb E[\Psi_{k+1}\mid\mathcal F_k]-\Psi_k\le(A_{k+1}-A_k)\Big(-\frac12+\theta_k+2\theta_k^2+2\theta_k^3\Big)\Big[\frac1L\mathbb E[\|\nabla f(x^k)\|^2\mid\mathcal F_k]+\frac{L}{A_k^5}\|p^k\|^2\Big].
\end{equation*}
Finally, $\theta_k\le1/8$ implies $\theta_k+2\theta_k^2+2\theta_k^3\le\frac18+\frac1{32}+\frac1{256}=\frac{41}{256}<\frac14$. Therefore, we conclude
\begin{align*}
    \mathbb E[\Psi_{k+1}\mid\mathcal F_k]-\Psi_k
    \le \ &-\frac{A_{k+1}-A_k}{4}\Big(\frac1L\mathbb E[\|\nabla f(x^k)\|^2\mid\mathcal F_k]+\frac{L}{A_k^5}\|p^k\|^2\Big)\\
    = \ &-\frac{A_{k+1}-A_k}{4L}\mathbb E[\|\nabla f(x^k)\|^2\mid\mathcal F_k]-\frac{L(A_{k+1}-A_k)}{4A_k^5}\|p^k\|^2,
\end{align*}
which completes the proof.
\end{proof}

\subsection{Convergence guarantees} We now establish the convergence rate in terms of the objective function value gap by combining the potential decrease in Lemma~\ref{lem:local} with a bound on the error from the deterministic endpoint $\xi_{k-1}(A_k)$ to the sampled iterate $x^k=\xi_{k-1}(u_k)$.

\begin{lemma}\label{lem:anytime-global}
Suppose $A_k\ge1$ and $\theta_k=(A_{k+1}-A_k)/A_k^{1/4}\le1/8$ for every $k\ge0$. Then, for every $k\ge0$, we have
\begin{equation*}
    \mathbb E[\Psi_k]\le\frac32L\|x^0-x^\star\|^2, \ \text{and} \ 
    f(\xi_{k-1}(s))-f^\star\le\frac{2\Psi_k}{A_k}, \ \forall s\in[A_k,A_{k+1}].
\end{equation*}
\end{lemma}

\begin{proof}
Since $A_0=1$, $\xi_{-1}(A_0)=x^0$, and $p^0=0$, we have $\Psi_0=f(x^0)-f^\star+L\|x^0-x^\star\|^2\le\frac32L\|x^0-x^\star\|^2$. Taking expectations in Lemma~\ref{lem:local} and using the nonnegativity of the two subtracted terms yields $\mathbb E[\Psi_{k+1}]\le\mathbb E[\Psi_k]$. Therefore, by induction, we obtain
\begin{equation*}
    \mathbb E[\Psi_k]\le\Psi_0\le\frac32L\|x^0-x^\star\|^2, \ \forall k\ge0.
\end{equation*}
For the objective function value along the curve, note that $\frac{p^k}{A_k^2}=(\xi_{k-1}(A_k)+\frac{p^k}{A_k^2}-x^\star)-(\xi_{k-1}(A_k)-x^\star)$. By the definition of the potential function \eqref{eq:anytime-energy}, we have
\begin{align*}
    \frac{\|p^k\|^2}{A_k^4}
    \le\ &2\|\xi_{k-1}(A_k)+\frac{p^k}{A_k^2}-x^\star\|^2
    +2\|\xi_{k-1}(A_k)-x^\star\|^2
    \le\frac{4\Psi_k}{L}.
\end{align*}
Moreover, for every $s\in[A_k,A_{k+1}]$, the curve satisfies
\begin{equation*}
    \|\xi_{k-1}(s)-\xi_{k-1}(A_k)\|
    =\frac12\Big(\frac1{A_k^2}-\frac1{s^2}\Big)\|p^k\|
    \le\frac{A_{k+1}-A_k}{A_k^3}\|p^k\|.
\end{equation*}
By applying the Lipschitz property \eqref{eq:Lipschitz-square-root-optimality-gap} with $x=\xi_{k-1}(A_k)$ and $y=\xi_{k-1}(s)$, and using $A_k[f(\xi_{k-1}(A_k))-f^\star]\le\Psi_k$, we obtain
\begin{align*}
    f(\xi_{k-1}(s))-f^\star
    \le \Big(\sqrt{\frac{\Psi_k}{A_k}}+\sqrt{\frac L2}\frac{A_{k+1}-A_k}{A_k}\frac{\|p^k\|}{A_k^2}\Big)^2 \le \frac{\Psi_k}{A_k}\Big(1+\frac{\sqrt2(A_{k+1}-A_k)}{\sqrt{A_k}}\Big)^2.
\end{align*}
Finally, since $A_k\ge1$ and $\theta_k\le1/8$, we have $(A_{k+1}-A_k)/\sqrt{A_k}=\theta_k/A_k^{1/4}\le1/8$ and $(1+\sqrt2/8)^2<2$. Consequently, we have
\begin{equation*}
    f(\xi_{k-1}(s))-f^\star\le\frac{2\Psi_k}{A_k}, \ \forall s\in[A_k,A_{k+1}],
\end{equation*}
and the proof is complete.
\end{proof}

Note that $x^K=\xi_{K-1}(u_K)$ with $u_K\in[A_K,A_{K+1}]$. Applying Lemma~\ref{lem:anytime-global} at $s=u_K$ yields
\begin{equation}\label{eq:anytime-general-weights}
    \mathbb E[f(x^K)-f^\star]\le\frac{2\mathbb E[\Psi_K]}{A_K}\le\frac{3L\|x^0-x^\star\|^2}{A_K}.
\end{equation}
The bound \eqref{eq:anytime-general-weights} suggests making $A_K$ large while keeping $(A_{k+1}-A_k)/A_k^{1/4}\le1/8$ at every iteration. To relate these two requirements, the concavity of $s^{3/4}$ gives
\begin{equation*}
    A_{k+1}^{3/4}-A_k^{3/4}\le\frac{3(A_{k+1}-A_k)}{4A_k^{1/4}}\le\frac3{32}.
\end{equation*}
Summing over $k=0,\ldots,K-1$ yields $A_K\le(1+3K/32)^{4/3}$. Therefore, the largest growth order allowed by this local condition is $A_K=\cO(K^{4/3})$. This limitation concerns the stated sufficient condition and does not prove optimality among all deterministic-boundary schedules. Section~\ref{sec:random-boundaries} avoids this condition by using a different treatment of the signed discrepancy. 

\begin{theorem}\label{thm:convex}
Let $f$ be convex and $L$-smooth, and let $x^\star$ be a minimizer. Choose $A_k=(1+k/12)^{4/3}$. Independently sample $u_k\sim\operatorname{Unif}(A_k,A_{k+1})$ for every $k\ge0$, and use the schedule \eqref{eq:anytime-parameters}. Then, for every deterministic $K\ge1$, the heavy-ball method satisfies
\begin{equation*}
    \mathbb E[f(x^K)-f^\star]\le\frac{84L\|x^0-x^\star\|^2}{K^{4/3}}.
\end{equation*}
\end{theorem}

\begin{proof}
The deterministic anytime grid properties in Lemma~\ref{lem:technical-anytime-grid} give
$\theta_k<1/8$ at every iteration and $A_K\ge K^{4/3}/12^{4/3}$.
Thus the condition of Lemma~\ref{lem:local} is satisfied, and \eqref{eq:anytime-general-weights} gives
\begin{align*}
    \mathbb E[f(x^K)-f^\star]
    \le\frac{3L\|x^0-x^\star\|^2}{A_K}
    \le\frac{3\cdot12^{4/3}L\|x^0-x^\star\|^2}{K^{4/3}}
    \le\frac{84L\|x^0-x^\star\|^2}{K^{4/3}},
\end{align*}
which completes the proof.
\end{proof}

The conditional potential decrease \eqref{eq:anytime-local} used in the expectation analysis also yields an almost sure last-iterate convergence rate along the same infinite trajectory.

\begin{corollary}\label{cor:anytime-almost-sure}
Under the conditions of Theorem~\ref{thm:convex}, the heavy-ball method with the schedule \eqref{eq:anytime-parameters} satisfies
\begin{equation*}
    \mathbb P\Big(\sup_{K\ge1}K^{4/3}[f(x^K)-f^\star]<\infty\Big)=1.
\end{equation*}
In particular, $f(x^K)-f^\star=\mathcal O(K^{-4/3})$ almost surely along the same infinite trajectory.
\end{corollary}

\begin{proof}
By construction, $\Psi_k$ is nonnegative and $\mathcal F_k$-measurable for every $k\ge0$. Under the conditions of Theorem~\ref{thm:convex}, Lemma~\ref{lem:local} gives $\mathbb E[\Psi_{k+1}\mid\mathcal F_k]\le\Psi_k$ for every $k\ge0$. Moreover, since $A_0=1$ and $p^0=0$, the definition of the potential and the $L$-smoothness of $f$ imply
\begin{equation*}
    \Psi_0=f(x^0)-f^\star+L\|x^0-x^\star\|^2\le\frac32L\|x^0-x^\star\|^2<\infty.
\end{equation*}
By taking expectations and iterating the inequality $\mathbb E[\Psi_{k+1}\mid\mathcal F_k]\le\Psi_k$, we have $0\le\mathbb E[\Psi_k]\le\Psi_0<\infty$ for every $k\ge0$. Consequently, $\{\Psi_k\}_{k\ge0}$ is a nonnegative supermartingale with respect to $\{\mathcal F_k\}_{k\ge0}$. By the nonnegative supermartingale convergence theorem (Theorem~4.2.12 of \citet{durrett2019probability}), there exists an almost surely finite nonnegative random variable $\Psi_\infty$ such that $\Psi_k\to\Psi_\infty$ almost surely. Since every convergent real sequence is bounded, it follows that $\sup_{k\ge0}\Psi_k<\infty$ almost surely.

On the other hand, the pathwise curve bound $f(\xi_{k-1}(s))-f^\star\le2\Psi_k/A_k$ established in Lemma~\ref{lem:anytime-global}, together with $x^K=\xi_{K-1}(u_K)$ and $u_K\in[A_K,A_{K+1}]$, gives $0\le f(x^K)-f^\star\le\frac{2\Psi_K}{A_K}$ for every $K\ge1$. For the schedule $A_K=(1+K/12)^{4/3}$, we have $A_K\ge K^{4/3}/12^{4/3}$. Therefore, on the probability-one event $\{\sup_{k\ge0}\Psi_k<\infty\}$, we have
\begin{equation*}
    \sup_{K\ge1}K^{4/3}[f(x^K)-f^\star]\le2\cdot12^{4/3}\sup_{k\ge0}\Psi_k<\infty,
\end{equation*}
which completes the proof.
\end{proof}

\section{Randomized Time Boundaries: Anytime Convergence}\label{sec:random-boundaries}
We now randomize the time boundaries to improve the anytime rate from $\mathcal O(K^{-4/3})$ to $\mathcal O(K^{-3/2})$. We retain the interpolation and potential from Section~\ref{sec:anytime-proof}, using the grid \eqref{eq:random-grid}, the evaluation times \eqref{eq:random-sampling}, and the filtration \eqref{eq:random-filtration}. All $V_k\sim\operatorname{Unif}[1,2]$ and $U_k\sim\operatorname{Unif}[0,1]$ are mutually independent, and all boundaries and coefficients are predefined and independent of the oracle. Equivalently, for $k\ge0$ the parameters are
\begin{equation}\label{eq:random-parameters}
 \delta_k=\tfrac12(u_k^{-2}-u_{k+1}^{-2}),\ 
 \eta_k=\frac{(A_{k+1}-A_k)u_k^2\delta_k}{L},\ 
 \beta_k=\frac{\delta_k}{\delta_{k-1}}\quad(k\ge1),
\end{equation}
with $\beta_0=0$. The almost sure positivity of $\delta_k$ follows from $u_k<A_{k+1}\le u_{k+1}$.

\begin{theorem}\label{thm:random-anytime}
Let $f$ be convex and $L$-smooth, and let $x^\star$ be a minimizer. Initialize $x^{-1}=x^0$ and run the heavy-ball method \eqref{eq:hb} with the schedule \eqref{eq:random-grid}--\eqref{eq:random-coefficients}. Then, for every deterministic $K\ge1$,
\begin{equation}\label{eq:random-anytime-rate}
 \E[f(x^K)-f^\star]
 \le\frac{4L\|x^0-x^\star\|^2}{(1+K/1024)^{3/2}}
 \le\frac{131072L\|x^0-x^\star\|^2}{K^{3/2}}.
\end{equation}
Along the same infinite trajectory,
\begin{equation}\label{eq:random-almost-sure-rate}
 \mathbb P\Big(\sup_{K\ge1}K^{3/2}[f(x^K)-f^\star]<\infty\Big)=1.
\end{equation}
In particular, $f(x^K)-f^\star=\mathcal O(K^{-3/2})$ almost surely.
\end{theorem}

\subsection{Inherited interpolation and potential function decrease}\label{subsec:random-discrepancy}
The rescaled update \eqref{eq:rescaled-two-step} and \eqref{eq:random-parameters} give
\[
 p^{k+1}=p^k-\frac{(A_{k+1}-A_k)u_k^2}{L}\nabla f(x^k).
\]
The curves remain those in \eqref{eq:anytime-curve}; in particular,
\begin{equation}\label{eq:random-curve-change}
 \xi_k(s)-\xi_{k-1}(s)
 =-\frac{A_{k+1}-A_k}{2L}\Big(1-\frac{u_k^2}{s^2}\Big)\nabla f(x^k),
 \quad s\ge u_k.
\end{equation}
The endpoint identities \eqref{eq:anytime-boundary} and Figure~\ref{fig:anytime-correction} hold pathwise. The anytime potential \eqref{eq:anytime-energy}, evaluated at the random left endpoint, is
\[
\begin{aligned}
 \Psi_k=A_k[f(\xi_{k-1}(A_k))-f^\star]
       +\frac L2\|\xi_{k-1}(A_k)-x^\star\|^2+\frac L2\|\xi_{k-1}(A_k)+ \frac{p^k}{A_k^2}-x^\star\|^2.
\end{aligned}
\]
Both $\Psi_k$ and the incoming curve are $\mathcal G_k$-measurable.

\begin{lemma}\label{lem:random-gradient-bounds}
For every $s,t\ge A_k$, it holds that
\begin{equation}\label{eq:random-gradient-difference}
 \|\nabla f(\xi_{k-1}(s))-\nabla f(\xi_{k-1}(t))\|
 \le\frac{L|s-t|}{A_k^3}\|p^k\|.
\end{equation}
Moreover, we have
\begin{equation}\label{eq:random-momentum-control}
 \frac{L\|p^k\|^2}{A_k^4}\le4\Psi_k, \  \|\nabla f(\xi_{k-1}(A_k))\|^2\le\frac{2L\Psi_k}{A_k}.
\end{equation}
\end{lemma}
\begin{proof}
Integrating $\xi_{k-1}'(s)=p^k/s^3$ gives
$\|\xi_{k-1}(s)-\xi_{k-1}(t)\|\le |s-t|\|p^k\|/A_k^3$.
Lipschitz continuity of the gradient proves \eqref{eq:random-gradient-difference}.
The squared-sum inequality gives the momentum bound directly:
\[
\begin{aligned}
 \frac{L\|p^k\|^2}{A_k^4} \le2L\|\xi_{k-1}(A_k)-x^\star\|^2+2L\|\xi_{k-1}(A_k)+p^k/A_k^2-x^\star\|^2\le4\Psi_k.
\end{aligned}
\]
The gradient bound follows from \eqref{eq:self-bounding-inequality} and $A_k[f(\xi_{k-1}(A_k))-f^\star] \le \Psi_k$.
\end{proof}

The incoming curve and the endpoints are fixed, and $u_k$ is uniform on $[A_k,A_{k+1}]$ conditionally on $(\mathcal G_k,V_k)$. Thus Lemma~\ref{lem:anytime-energy-defect} applies. The following bound refines the inequality \eqref{eq:anytime-defect} in Lemma~\ref{lem:anytime-energy-defect}.

\begin{lemma}\label{lem:random-signed-defect}
For the grid \eqref{eq:random-grid}, we have
\begin{equation}\label{eq:random-signed-defect}
\begin{aligned}
 &\E[\Psi_{k+1}\mid\mathcal G_k,V_k]-\Psi_k \\
 \le\ &\frac1{2L}\int_{A_k}^{A_{k+1}}(s-A_k)\|\nabla f(\xi_{k-1}(s))\|^2\dd s -\frac{(A_{k+1}-A_k)^2}{4L}\|\nabla f(\xi_{k-1}(A_{k+1}))\|^2 \\
 &-\frac{63(A_{k+1}-A_k)}{128L} \Big(\E[\|\nabla f(x^k)\|^2\mid\mathcal G_k,V_k]+\frac{L^2\|p^k\|^2}{A_k^5}\Big).
\end{aligned}
\end{equation}
\end{lemma}
\begin{proof}
First, note that $A_k$, $A_{k+1}$, $p^k$, and the incoming curve $\xi_{k-1}$ are fixed conditionally on $(\mathcal G_k,V_k)$, and $u_k$ is uniform on $[A_k,A_{k+1}]$. Expanding the third term on the right-hand side of \eqref{eq:anytime-defect} gives
\begin{equation}\label{eq:randomized-potential-terms}
    \begin{aligned}
         &\E[\Psi_{k+1}\mid\mathcal G_k,V_k]-\Psi_k \\
         \le \ &-\frac{A_{k+1}-A_k}{2L}
       \E[\|\nabla f(x^k)\|^2\mid\mathcal G_k,V_k]
       -\frac L4(A_k^{-4}-A_{k+1}^{-4})\|p^k\|^2\\
 &+\frac{(A_{k+1}-A_k)^2}{2L}
       \E[\|\nabla f(x^k)\|^2\mid\mathcal G_k,V_k] \\
       &-\frac{A_{k+1}-A_k}{L}
       \E[(A_{k+1}-u_k)
       \langle\nabla f(\xi_{k-1}(A_{k+1})),\nabla f(x^k)\rangle
       \mid\mathcal G_k,V_k]\\
 &+\frac{2(A_{k+1}-A_k)^2}{A_k^3}\|p^k\|
       \E[\|\nabla f(x^k)\|\mid\mathcal G_k,V_k]\\
 &+\frac{(A_{k+1}-A_k)^4}{2LA_k}
       \E[\|\nabla f(x^k)\|^2\mid\mathcal G_k,V_k]\\
 &+\frac{(A_{k+1}-A_k)^3}{LA_k}
       \|\nabla f(\xi_{k-1}(A_{k+1}))\|
       \E[\|\nabla f(x^k)\|\mid\mathcal G_k,V_k].
    \end{aligned}
\end{equation}
Note that $2\langle\nabla f(\xi_{k-1}(A_{k+1})),\nabla f(x^k)\rangle =\|\nabla f(\xi_{k-1}(A_{k+1}))\|^2+\|\nabla f(x^k)\|^2-\|\nabla f(x^k)-\nabla f(\xi_{k-1}(A_{k+1}))\|^2$. Substituting this identity yields
\begin{align*}
 &\frac{(A_{k+1}-A_k)^2}{2L}
       \E[\|\nabla f(x^k)\|^2\mid\mathcal G_k,V_k] -\frac{A_{k+1}-A_k}{L}
       \E[(A_{k+1}-u_k)
       \langle\nabla f(\xi_{k-1}(A_{k+1})),\nabla f(x^k)\rangle
       \mid\mathcal G_k,V_k]\\
 = \ &\frac{(A_{k+1}-A_k)^2}{2L}
       \E[\|\nabla f(x^k)\|^2\mid\mathcal G_k,V_k]-\frac{A_{k+1}-A_k}{2L}
       \E[(A_{k+1}-u_k)\|\nabla f(x^k)\|^2
       \mid\mathcal G_k,V_k]\\
 &-\frac{A_{k+1}-A_k}{2L}
       \E[A_{k+1}-u_k\mid\mathcal G_k,V_k]
       \|\nabla f(\xi_{k-1}(A_{k+1}))\|^2\\
 &+\frac{A_{k+1}-A_k}{2L}
       \E[(A_{k+1}-u_k)
       \|\nabla f(x^k)-\nabla f(\xi_{k-1}(A_{k+1}))\|^2
       \mid\mathcal G_k,V_k]\\
 = \ &\frac{A_{k+1}-A_k}{2L}
       \E[(u_k-A_k)\|\nabla f(x^k)\|^2
       \mid\mathcal G_k,V_k]-\frac{A_{k+1}-A_k}{2L}
       \E[A_{k+1}-u_k\mid\mathcal G_k,V_k]
       \|\nabla f(\xi_{k-1}(A_{k+1}))\|^2\\
 &+\frac{A_{k+1}-A_k}{2L}
       \E[(A_{k+1}-u_k)
       \|\nabla f(x^k)-\nabla f(\xi_{k-1}(A_{k+1}))\|^2
       \mid\mathcal G_k,V_k].
\end{align*}
Since $u_k$ is conditionally uniform on $[A_k,A_{k+1}]$ and $x^k=\xi_{k-1}(u_k)$, it follows
\begin{align*}
 \E[(u_k-A_k)\|\nabla f(x^k)\|^2\mid\mathcal G_k,V_k]
 &=\frac1{A_{k+1}-A_k}
   \int_{A_k}^{A_{k+1}}(s-A_k)
   \|\nabla f(\xi_{k-1}(s))\|^2\dd s,\\
 \E[A_{k+1}-u_k\mid\mathcal G_k,V_k]
 &=\frac{A_{k+1}-A_k}{2}.
\end{align*}
Therefore, we obtain
\begin{align*}
 &\frac{(A_{k+1}-A_k)^2}{2L}
       \E[\|\nabla f(x^k)\|^2\mid\mathcal G_k,V_k] \\
       &-\frac{A_{k+1}-A_k}{L}
       \E[(A_{k+1}-u_k)
       \langle\nabla f(\xi_{k-1}(A_{k+1})),\nabla f(x^k)\rangle
       \mid\mathcal G_k,V_k]\\
 =\ &\frac1{2L}\int_{A_k}^{A_{k+1}}(s-A_k)
       \|\nabla f(\xi_{k-1}(s))\|^2\dd s
       -\frac{(A_{k+1}-A_k)^2}{4L}
       \|\nabla f(\xi_{k-1}(A_{k+1}))\|^2\\
 &+\frac{A_{k+1}-A_k}{2L}
       \E[(A_{k+1}-u_k)
       \|\nabla f(x^k)-\nabla f(\xi_{k-1}(A_{k+1}))\|^2
       \mid\mathcal G_k,V_k].
\end{align*}
Taking $s=u_k$ and $t=A_{k+1}$ in \eqref{eq:random-gradient-difference} gives
\[
 \|\nabla f(x^k)-\nabla f(\xi_{k-1}(A_{k+1}))\|
 \le\frac{L(A_{k+1}-u_k)}{A_k^3}\|p^k\|,
\]
which implies
\begin{align*}
    \begin{aligned}
        &\frac{A_{k+1}-A_k}{2L}
       \E[(A_{k+1}-u_k)
       \|\nabla f(x^k)-\nabla f(\xi_{k-1}(A_{k+1}))\|^2
       \mid\mathcal G_k,V_k]\\
 \le \ &\frac{L(A_{k+1}-A_k)}{2A_k^6}\|p^k\|^2
       \E[(A_{k+1}-u_k)^3\mid\mathcal G_k,V_k]\\
 = \ &\frac{L(A_{k+1}-A_k)^4}{8A_k^6}\|p^k\|^2\\
 = \ &\frac{A_{k+1}-A_k}{L}
       \frac{(A_{k+1}-A_k)^3}{8A_k}
       \frac{L^2\|p^k\|^2}{A_k^5}\\
 \le\ &\frac{A_{k+1}-A_k}{512^3L}
       \frac{L^2\|p^k\|^2}{A_k^5}, 
    \end{aligned}
\end{align*}
where the equality uses $\E[(A_{k+1}-u_k)^3\mid\mathcal G_k,V_k]=(A_{k+1}-A_k)^3/4$ from Lemma~\ref{lem:technical-interval-calculus}, and the last inequality follows from
$(A_{k+1}-A_k)^3/A_k\le1/256^3$ in \eqref{eq:random-local-scales}. To this end, we obtain
\begin{equation}\label{eq:randomized-boundary-I}
    \begin{aligned}
        &\frac{(A_{k+1}-A_k)^2}{2L}
       \E[\|\nabla f(x^k)\|^2\mid\mathcal G_k,V_k] \\
       &-\frac{A_{k+1}-A_k}{L}
       \E[(A_{k+1}-u_k)
       \langle\nabla f(\xi_{k-1}(A_{k+1})),\nabla f(x^k)\rangle
       \mid\mathcal G_k,V_k]\\
 \le \ &\frac1{2L}\int_{A_k}^{A_{k+1}}(s-A_k)
       \|\nabla f(\xi_{k-1}(s))\|^2\dd s
       -\frac{(A_{k+1}-A_k)^2}{4L}
       \|\nabla f(\xi_{k-1}(A_{k+1}))\|^2 \\
       &+\frac{A_{k+1}-A_k}{512^3L}
       \frac{L^2\|p^k\|^2}{A_k^5}.
    \end{aligned}
\end{equation}
We now bound the last three terms in \eqref{eq:randomized-potential-terms}. For the first one of these three terms, we have
\begin{align*}
 \frac{2(A_{k+1}-A_k)^2}{A_k^3}\|p^k\|\|\nabla f(x^k)\| = \ &\frac{A_{k+1}-A_k}{L}
   \frac{A_{k+1}-A_k}{\sqrt{A_k}}
   2\|\nabla f(x^k)\|\frac{L\|p^k\|}{A_k^{5/2}}\\
 \le \ &\frac{A_{k+1}-A_k}{L}
   \frac{A_{k+1}-A_k}{\sqrt{A_k}}
   \left(\|\nabla f(x^k)\|^2+
   \frac{L^2\|p^k\|^2}{A_k^5}\right)\\
 \le \ &\frac{A_{k+1}-A_k}{256L}
   \left(\|\nabla f(x^k)\|^2+
   \frac{L^2\|p^k\|^2}{A_k^5}\right),
\end{align*}
where we used Young's inequality and \eqref{eq:random-local-scales}. Taking conditional expectations gives
\begin{equation}\label{eq:randomized-boundary-II}
  \frac{2(A_{k+1}-A_k)^2}{A_k^3}\|p^k\|
   \E[\|\nabla f(x^k)\|\mid\mathcal G_k,V_k] \le\frac{A_{k+1}-A_k}{256L}
   \left(\E[\|\nabla f(x^k)\|^2\mid\mathcal G_k,V_k]
   +\frac{L^2\|p^k\|^2}{A_k^5}\right).  
\end{equation}
For the second term of these three terms, \eqref{eq:random-local-scales} gives
\begin{align*}
 \frac{(A_{k+1}-A_k)^4}{2LA_k}\|\nabla f(x^k)\|^2= \ &\frac{A_{k+1}-A_k}{2L}\frac{(A_{k+1}-A_k)^3}{A_k}\|\nabla f(x^k)\|^2 \\
 \le \ &\frac{A_{k+1}-A_k}{2\cdot256^3L}\|\nabla f(x^k)\|^2 \\
 = \ &\frac{4(A_{k+1}-A_k)}{512^3L}\|\nabla f(x^k)\|^2.
\end{align*}
Therefore, we obtain
\begin{equation}\label{eq:randomized-boundary-III}
   \frac{(A_{k+1}-A_k)^4}{2LA_k}
 \E[\|\nabla f(x^k)\|^2\mid\mathcal G_k,V_k]
 \le\frac{4(A_{k+1}-A_k)}{512^3L}
 \E[\|\nabla f(x^k)\|^2\mid\mathcal G_k,V_k].  
\end{equation}
For the last term, by using \eqref{eq:random-gradient-difference} and $A_{k+1}-u_k\le A_{k+1}-A_k$, we derive
\[
 \|\nabla f(\xi_{k-1}(A_{k+1}))\| \le\|\nabla f(x^k)\|+\frac{L(A_{k+1}-A_k)}{A_k^3}\|p^k\|,
\]
which further implies
\begin{align*}
 &\frac{(A_{k+1}-A_k)^3}{LA_k}
   \|\nabla f(\xi_{k-1}(A_{k+1}))\|\|\nabla f(x^k)\|\\
 \le \ &\frac{(A_{k+1}-A_k)^3}{LA_k}\|\nabla f(x^k)\|^2
   +\frac{(A_{k+1}-A_k)^4}{A_k^4}\|p^k\|\|\nabla f(x^k)\|\\
 = \ &\frac{A_{k+1}-A_k}{L}
   \left[
   \frac{(A_{k+1}-A_k)^2}{A_k}\|\nabla f(x^k)\|^2
   +\frac{(A_{k+1}-A_k)^3}{A_k^{3/2}}
    \|\nabla f(x^k)\|\frac{L\|p^k\|}{A_k^{5/2}}
   \right]\\
 \le \ &\frac{A_{k+1}-A_k}{L}
   \left[
   \frac1{65536}\|\nabla f(x^k)\|^2
   +\frac1{2\cdot256^3}
    \left(\|\nabla f(x^k)\|^2+
          \frac{L^2\|p^k\|^2}{A_k^5}\right)
   \right]\\
 \le \ &\frac{A_{k+1}-A_k}{L}
   \left(\frac1{65536}+\frac4{512^3}\right)
   \left(\|\nabla f(x^k)\|^2+
         \frac{L^2\|p^k\|^2}{A_k^5}\right).
\end{align*}
The above constants follows from $(A_{k+1}-A_k)^2/A_k\le1/256^2=1/65536$ and
$(A_{k+1}-A_k)^3/A_k^{3/2}\le1/256^3$. Taking conditional expectations yields
\begin{equation}\label{eq:randomized-boundary-IV}
    \begin{aligned}
        &\frac{(A_{k+1}-A_k)^3}{LA_k}
   \|\nabla f(\xi_{k-1}(A_{k+1}))\|
   \E[\|\nabla f(x^k)\|\mid\mathcal G_k,V_k] \\
   \le \ &\frac{A_{k+1}-A_k}{L}
   \left(\frac1{65536}+\frac4{512^3}\right)
   \left(\E[\|\nabla f(x^k)\|^2\mid\mathcal G_k,V_k]
         +\frac{L^2\|p^k\|^2}{A_k^5}\right). 
    \end{aligned}
\end{equation}
By combining the four inequalities \eqref{eq:randomized-boundary-I}, \eqref{eq:randomized-boundary-II}, \eqref{eq:randomized-boundary-III}, and \eqref{eq:randomized-boundary-IV}, we obtain
\begin{equation}\label{eq:randomized-boundary-four-combined}
\begin{aligned}
    &\E[\Psi_{k+1}\mid\mathcal G_k,V_k]-\Psi_k \\
    \le \ &-\frac{A_{k+1}-A_k}{2L}
       \E[\|\nabla f(x^k)\|^2\mid\mathcal G_k,V_k]
       -\frac L4(A_k^{-4}-A_{k+1}^{-4})\|p^k\|^2\\
       &+\frac1{2L}\int_{A_k}^{A_{k+1}}(s-A_k)
       \|\nabla f(\xi_{k-1}(s))\|^2\dd s
       -\frac{(A_{k+1}-A_k)^2}{4L}
       \|\nabla f(\xi_{k-1}(A_{k+1}))\|^2 \\
       &+\frac{A_{k+1}-A_k}{128L}
 \left(\E[\|\nabla f(x^k)\|^2\mid\mathcal G_k,V_k]
       +\frac{L^2\|p^k\|^2}{A_k^5}\right).
\end{aligned}
\end{equation}
It remains to bound the two negative terms in \eqref{eq:randomized-boundary-four-combined}. By \eqref{eq:random-momentum-dissipation}, we have
\begin{align*}
 -\frac L4(A_k^{-4}-A_{k+1}^{-4})\|p^k\|^2 \le-\frac{L(A_{k+1}-A_k)}{2A_k^5}\|p^k\|^2 =-\frac{A_{k+1}-A_k}{2L}\frac{L^2\|p^k\|^2}{A_k^5}.
\end{align*}
Consequently, it follows
\begin{align*}
 &-\frac{A_{k+1}-A_k}{2L}
   \E[\|\nabla f(x^k)\|^2\mid\mathcal G_k,V_k]
   -\frac L4(A_k^{-4}-A_{k+1}^{-4})\|p^k\|^2 \\
   \le \ &-\frac{A_{k+1}-A_k}{2L} \left(\E[\|\nabla f(x^k)\|^2\mid\mathcal G_k,V_k]+\frac{L^2\|p^k\|^2}{A_k^5}\right).
\end{align*}
Substituting the above inequality to \eqref{eq:randomized-boundary-four-combined} completes the proof.
\end{proof}

To proceed, we now handle the term 
\begin{align*}
   \frac1{2L}\int_{A_k}^{A_{k+1}}(s-A_k)\|\nabla f(\xi_{k-1}(s))\|^2\dd s-\frac{(A_{k+1}-A_k)^2}{4L}\|\nabla f(\xi_{k-1}(A_{k+1}))\|^2 
\end{align*}
in the right-hand side of \eqref{eq:random-signed-defect}. \begin{lemma}\label{lem:random-perturbation}
For every bounded stopping time $\tau$ with respect to $\{\mathcal G_k\}_{k\ge0}$, we have
\begin{align*}
 &\E\Bigg[\sum_{k=0}^{\tau-1}\Big[\frac1{2L}\int_{A_k}^{A_{k+1}}(s-A_k)\|\nabla f(\xi_{k-1}(s))\|^2\dd s-\frac{(A_{k+1}-A_k)^2}{4L}\|\nabla f(\xi_{k-1}(A_{k+1}))\|^2\Big]\Bigg]\\
 \le \ &\frac18\Psi_0+\frac18\E[\Psi_\tau]+\frac1{4L}\E\sum_{k=0}^{\tau-1}(A_{k+1}-A_k)\Big(\|\nabla f(x^k)\|^2+\frac{L^2\|p^k\|^2}{A_k^5}\Big).
\end{align*}
\end{lemma}

\begin{proof}
Conditionally on $\mathcal G_k$, the incoming curve and $A_k$ are fixed. By changing variables to $s=A_k+A_k^{1/3}t/512$ and integrating with respect to $V_k$, we have
\begin{align*}
 &\E\Big[\frac1{2L}\int_{A_k}^{A_{k+1}}(s-A_k)\|\nabla f(\xi_{k-1}(s))\|^2\dd s-\frac{(A_{k+1}-A_k)^2}{4L}\|\nabla f(\xi_{k-1}(A_{k+1}))\|^2\mid\mathcal G_k\Big]\\
 = \ &\frac{A_k^{2/3}}{512^2L}\Big[\int_0^1\frac t2\|\nabla f(\xi_{k-1}(A_k+\tfrac{A_k^{1/3}t}{512}))\|^2\dd t+\int_1^2\Big(t-\frac{3t^2}{4}\Big)\|\nabla f(\xi_{k-1}(A_k+\tfrac{A_k^{1/3}t}{512}))\|^2\dd t\Big]\\
 = \ &\frac{A_k^{2/3}}{512^2L}\E\Big[\int_0^\infty\kappa(t)\Big(\|\nabla f(\xi_{k-1}(A_k+\tfrac{A_k^{1/3}t}{512}))\|^2-\|\nabla f(\xi_{k-1}(A_{k+1}+\tfrac{A_k^{1/3}t}{512}))\|^2\Big)\dd t\mid\mathcal G_k\Big].
\end{align*}
By \eqref{eq:random-gradient-difference}, we have
\begin{equation*}
 \|\nabla f(\xi_{k-1}(A_k+\tfrac{A_k^{1/3}t}{512}))\|\le\|\nabla f(\xi_{k-1}(A_k))\|+\frac{Lt\|p^k\|}{512A_k^{8/3}},
\end{equation*}
which implies
\begin{align*}
    \|\nabla f(\xi_{k-1}(A_k+\tfrac{A_k^{1/3}t}{512}))\|^2 \le2\|\nabla f(\xi_{k-1}(A_k))\|^2+2\Bigg(\frac{Lt\|p^k\|}{512A_k^{8/3}}\Bigg)^2 \le \Big(\frac{4L}{A_k} + \frac{8Lt^2}{512^2A_k^{4/3}}\Big)\Psi_k,
\end{align*}
where \eqref{eq:random-momentum-control} is used. Consequently, we obtain
\begin{align*}
    \frac{A_k^{2/3}}{512^2L}\int_0^\infty|\kappa(t)|\|\nabla f(\xi_{k-1}(A_k+\tfrac{A_k^{1/3}t}{512}))\|^2\dd t \le \Big(\frac4{512^2A_k^{1/3}}+\frac8{512^4A_k^{2/3}}\Big)\Psi_k\int_0^\infty(1+t)^2|\kappa(t)|\dd t\le\frac18\Psi_k.
\end{align*}
It remains to control the change of curve between consecutive terms. For $t\ge0$ and $A_{k+1}\le s\le A_{k+1}+A_{k+1}^{1/3}t/512$, \eqref{eq:random-local-scales} and \eqref{eq:random-scale-change} give $s-u_k\le A_k^{1/3}(1+t)/256$. By \eqref{eq:random-gradient-difference} and \eqref{eq:random-curve-change}, we have
\begin{align*}
 \|\nabla f(\xi_{k-1}(s))\|
 &\le\|\nabla f(x^k)\|+\frac{L(s-u_k)\|p^k\|}{A_k^3}\le\sqrt2(1+t)\Big(\|\nabla f(x^k)\|^2+\frac{L^2\|p^k\|^2}{A_k^5}\Big)^{1/2},\\
 \|\nabla f(\xi_k(s))-\nabla f(\xi_{k-1}(s))\|
 &\le\frac{(A_{k+1}-A_k)(s-u_k)}{A_k}\|\nabla f(x^k)\|\le\frac{1+t}{65536A_k^{1/3}}\|\nabla f(x^k)\|.
\end{align*}
In particular, it follows
\begin{align*}
   \|\nabla f(\xi_k(s))\|\le2\sqrt2(1+t)\Big(\|\nabla f(x^k)\|^2+\frac{L^2\|p^k\|^2}{A_k^5}\Big)^{1/2}. 
\end{align*}
Therefore, on the one hand, applying the triangle inequality gives
\begin{align*}
 &\Big\|A_{k+1}^{1/3}\nabla f(\xi_k(A_{k+1}+\tfrac{A_{k+1}^{1/3}t}{512}))-A_k^{1/3}\nabla f(\xi_{k-1}(A_{k+1}+\tfrac{A_k^{1/3}t}{512}))\Big\|\\
 \le \ &\Big\|(A_{k+1}^{1/3}-A_k^{1/3})\nabla f(\xi_k(A_{k+1}+\tfrac{A_{k+1}^{1/3}t}{512}))\Big\| + \Big\|A_k^{1/3}\big(\nabla f(\xi_k(A_{k+1}+\tfrac{A_{k+1}^{1/3}t}{512})) - \nabla f(\xi_{k-1}(A_{k+1}+\tfrac{A_{k+1}^{1/3}t}{512}))\big)\Big\| \\
 &+\Big\|A_k^{1/3}\big(\nabla f(\xi_{k-1}(A_{k+1}+\tfrac{A_{k+1}^{1/3}t}{512})) - \nabla f(\xi_{k-1}(A_{k+1}+\tfrac{A_{k}^{1/3}t}{512}))\big)\Big\|&\\
 \le \ &2\sqrt2(A_{k+1}^{1/3}-A_k^{1/3})(1+t)\Big(\|\nabla f(x^k)\|^2+\frac{L^2\|p^k\|^2}{A_k^5}\Big)^{1/2}+\frac{1+t}{65536}\|\nabla f(x^k)\|+\frac{L\|p^k\|(A_{k+1}^{1/3}-A_k^{1/3})t}{512A_k^{8/3}}\\
 \le \ &\frac{3(1+t)}{512}\Big(\|\nabla f(x^k)\|^2+\frac{L^2\|p^k\|^2}{A_k^5}\Big)^{1/2}.
\end{align*}
where the last inequality follows from \eqref{eq:random-scale-change} and $A_k \ge 1$. Since $A_{k+1}^{1/3}\le2A_k^{1/3}$, on the other hand, we have
\begin{align*}
    &A_{k+1}^{1/3}\|\nabla f(\xi_k(A_{k+1}+\tfrac{A_{k+1}^{1/3}t}{512}))\|+A_k^{1/3}\|\nabla f(\xi_{k-1}(A_{k+1}+\tfrac{A_k^{1/3}t}{512}))\| \\
    \le \ &8A_k^{1/3}(1+t)\Big(\|\nabla f(x^k)\|^2+\frac{L^2\|p^k\|^2}{A_k^5}\Big)^{1/2}.
\end{align*}
Therefore, we obtain
\begin{align*}
 &\frac1{512^2L}\Big|\int_0^\infty\kappa(t)\big(A_{k+1}^{2/3}\|\nabla f(\xi_k(A_{k+1}+\tfrac{A_{k+1}^{1/3}t}{512}))\|^2-A_k^{2/3}\|\nabla f(\xi_{k-1}(A_{k+1}+\tfrac{A_k^{1/3}t}{512}))\|^2\big)\dd t\Big|\\
 \le \ &\frac1{512^2L}\int_0^\infty|\kappa(t)|\cdot \Big|A_{k+1}^{2/3}\|\nabla f(\xi_k(A_{k+1}+\tfrac{A_{k+1}^{1/3}t}{512}))\|^2-A_k^{2/3}\|\nabla f(\xi_{k-1}(A_{k+1}+\tfrac{A_k^{1/3}t}{512}))\|^2\Big| \dd t\\
 \le \ &\frac1{512^2L}\int_0^\infty|\kappa(t)|\cdot \|A_{k+1}^{1/3}\nabla f(\xi_k(A_{k+1}+\tfrac{A_{k+1}^{1/3}t}{512}))-A_k^{1/3}\nabla f(\xi_{k-1}(A_{k+1}+\tfrac{A_k^{1/3}t}{512}))\| \\
 &\cdot \big(\|A_{k+1}^{1/3}\nabla f(\xi_k(A_{k+1}+\tfrac{A_{k+1}^{1/3}t}{512}))\| + \|A_k^{1/3}\nabla f(\xi_{k-1}(A_{k+1}+\tfrac{A_k^{1/3}t}{512}))\|\big) \dd t\\
 \le \ &\frac{24A_k^{1/3}}{512^3L}\Big(\|\nabla f(x^k)\|^2+\frac{L^2\|p^k\|^2}{A_k^5}\Big)\int_0^\infty(1+t)^2|\kappa(t)|\dd t\\
 \le \ &\frac{3A_k^{1/3}}{65536L}\Big(\|\nabla f(x^k)\|^2+\frac{L^2\|p^k\|^2}{A_k^5}\Big)\\
 \le \ &\frac{A_{k+1}-A_k}{4L}\Big(\|\nabla f(x^k)\|^2+\frac{L^2\|p^k\|^2}{A_k^5}\Big),
\end{align*}
where the last step uses $A_{k+1}-A_k\ge A_k^{1/3}/512$.

Finally, since $\tau$ is a stopping time with respect to $\{\mathcal G_k\}_{k\ge0}$, we have $\{k<\tau\}\in\mathcal G_k$. Therefore, using the tower property gives
\begin{align*}
    &\E\Big[\mathbf 1_{\{k<\tau\}}\Big(\frac1{2L}\int_{A_k}^{A_{k+1}}(s-A_k)\|\nabla f(\xi_{k-1}(s))\|^2\dd s-\frac{(A_{k+1}-A_k)^2}{4L}\|\nabla f(\xi_{k-1}(A_{k+1}))\|^2\Big)\Big]\\
    = \ &\frac1{512^2L}\E\Big[\mathbf 1_{\{k<\tau\}}A_k^{2/3}\int_0^\infty\kappa(t)\big(\|\nabla f(\xi_{k-1}(A_k+\frac{A_k^{1/3}t}{512}))\|^2-\|\nabla f(\xi_{k-1}(A_{k+1}+\frac{A_k^{1/3}t}{512}))\|^2\big)\dd t\Big].
\end{align*}
Since $\tau$ is bounded, summing over $k$ yields
\begin{align*}
    &\E\Big[\sum_{k=0}^{\tau-1}\big[\frac1{2L}\int_{A_k}^{A_{k+1}}(s-A_k)\|\nabla f(\xi_{k-1}(s))\|^2\dd s-\frac{(A_{k+1}-A_k)^2}{4L}\|\nabla f(\xi_{k-1}(A_{k+1}))\|^2\big]\Big]\\
    = \ &\frac1{512^2L}\E\sum_{k=0}^{\tau-1}\int_0^\infty\kappa(t)\Big[A_k^{2/3}\|\nabla f(\xi_{k-1}(A_k+\frac{A_k^{1/3}t}{512}))\|^2-A_k^{2/3}\|\nabla f(\xi_{k-1}(A_{k+1}+\frac{A_k^{1/3}t}{512}))\|^2\Big]\dd t,
\end{align*}
which further implies
\begin{align*}
    &\E\Big[\sum_{k=0}^{\tau-1}\big[\frac1{2L}\int_{A_k}^{A_{k+1}}(s-A_k)\|\nabla f(\xi_{k-1}(s))\|^2\dd s-\frac{(A_{k+1}-A_k)^2}{4L}\|\nabla f(\xi_{k-1}(A_{k+1}))\|^2\big]\Big]\\
    = \ &\frac1{512^2L}\E\Big[\sum_{k=0}^{\tau-1}\int_0^\infty\kappa(t)\big[A_k^{2/3}\|\nabla f(\xi_{k-1}(A_k+\frac{A_k^{1/3}t}{512}))\|^2-A_{k+1}^{2/3}\|\nabla f(\xi_k(A_{k+1}+\frac{A_{k+1}^{1/3}t}{512}))\|^2\big]\dd t\Big]\\
    &+\frac1{512^2L}\E\Big[\sum_{k=0}^{\tau-1}\int_0^\infty\kappa(t)\big[A_{k+1}^{2/3}\|\nabla f(\xi_k(A_{k+1}+\frac{A_{k+1}^{1/3}t}{512}))\|^2-A_k^{2/3}\|\nabla f(\xi_{k-1}(A_{k+1}+\frac{A_k^{1/3}t}{512}))\|^2\big]\dd t\Big].
\end{align*}
Since $A_0=1$ and $\xi_{-1}(s)=x^0$, it follows
\begin{align*}
     &\E\Big[\sum_{k=0}^{\tau-1}\big[\frac1{2L}\int_{A_k}^{A_{k+1}}(s-A_k)\|\nabla f(\xi_{k-1}(s))\|^2\dd s-\frac{(A_{k+1}-A_k)^2}{4L}\|\nabla f(\xi_{k-1}(A_{k+1}))\|^2\big]\Big]\\
    = \ &\frac1{512^2L}\int_0^\infty\kappa(t)\|\nabla f(x^0)\|^2\dd t-\E\Big[\frac{A_\tau^{2/3}}{512^2L}\int_0^\infty\kappa(t)\|\nabla f(\xi_{\tau-1}(A_\tau+\frac{A_\tau^{1/3}t}{512}))\|^2\dd t\Big]\\
    &+\frac1{512^2L}\E\Big[\sum_{k=0}^{\tau-1}\int_0^\infty\kappa(t)\big[A_{k+1}^{2/3}\|\nabla f(\xi_k(A_{k+1}+\frac{A_{k+1}^{1/3}t}{512}))\|^2-A_k^{2/3}\|\nabla f(\xi_{k-1}(A_{k+1}+\frac{A_k^{1/3}t}{512}))\|^2\big]\dd t\Big] \\
    \le \ &\frac18\Psi_0+\frac18\E[\Psi_\tau]+\frac1{4L}\E\sum_{k=0}^{\tau-1}(A_{k+1}-A_k)\Big(\|\nabla f(x^k)\|^2+\frac{L^2\|p^k\|^2}{A_k^5}\Big),
\end{align*}
which completes the proof.
\end{proof}

\begin{lemma}\label{lem:random-energy}
For every bounded stopping time $\tau$ with respect to $\{\mathcal G_k\}_{k\ge0}$, we have
\begin{equation*}
 \E[\Psi_\tau]+\frac1{2048L}\E\Big[\sum_{k=0}^{\tau-1}A_k^{1/3}\big(\|\nabla f(x^k)\|^2+\frac{L^2\|p^k\|^2}{A_k^5}\big)\Big]\le\frac97\Psi_0.
\end{equation*}
\end{lemma}
\begin{proof}
Since $\tau$ is a stopping time with respect to $\{\mathcal G_k\}_{k\ge0}$, we have $\{k<\tau\}\in\mathcal G_k$. Hence $\mathbf 1_{\{k<\tau\}}$ is also measurable with respect to $\sigma(\mathcal G_k,V_k)$. Multiplying \eqref{eq:random-signed-defect} by $\mathbf 1_{\{k<\tau\}}$ and taking expectations gives
\begin{align*}
    &\E\big[\mathbf 1_{\{k<\tau\}}\big(\E[\Psi_{k+1}\mid\mathcal G_k,V_k]-\Psi_k\big)\big]\\
    \le \ &\E\Big[\mathbf 1_{\{k<\tau\}}\Big(\frac1{2L}\int_{A_k}^{A_{k+1}}(s-A_k)\|\nabla f(\xi_{k-1}(s))\|^2\dd s-\frac{(A_{k+1}-A_k)^2}{4L}\|\nabla f(\xi_{k-1}(A_{k+1}))\|^2\Big)\Big]\\
    &-\frac{63}{128L}\E\Big[\mathbf 1_{\{k<\tau\}}(A_{k+1}-A_k)\Big(\E[\|\nabla f(x^k)\|^2\mid\mathcal G_k,V_k]+\frac{L^2\|p^k\|^2}{A_k^5}\Big)\Big].
\end{align*}
Since $\mathbf 1_{\{k<\tau\}}$ is $\sigma(\mathcal G_k,V_k)$-measurable, the tower property gives $\E[\mathbf 1_{\{k<\tau\}}\E[\Psi_{k+1}\mid\mathcal G_k,V_k]]=\E[\mathbf 1_{\{k<\tau\}}\Psi_{k+1}]$, which implies
\begin{equation*}
    \E\big[\mathbf 1_{\{k<\tau\}}\big(\E[\Psi_{k+1}\mid\mathcal G_k,V_k]-\Psi_k\big)\big]=\E[\mathbf 1_{\{k<\tau\}}(\Psi_{k+1}-\Psi_k)].
\end{equation*}
Moreover, since $A_{k+1}-A_k$ is fixed conditionally on $(\mathcal G_k,V_k)$ and $\mathbf 1_{\{k<\tau\}}$ is $\mathcal G_k$-measurable, we also have
\begin{equation*}
    \E\big[\mathbf 1_{\{k<\tau\}}(A_{k+1}-A_k)\E[\|\nabla f(x^k)\|^2\mid\mathcal G_k,V_k]\big]=\E\big[\mathbf 1_{\{k<\tau\}}(A_{k+1}-A_k)\|\nabla f(x^k)\|^2\big].
\end{equation*}
Since $\tau$ is bounded, there exists a deterministic integer $N$ such that $\tau\le N$ almost surely. Summing the preceding inequality over $k=0,\ldots,N-1$ gives
\begin{equation*}
    \sum_{k=0}^{N-1}\mathbf 1_{\{k<\tau\}}(\Psi_{k+1}-\Psi_k)=\sum_{k=0}^{\tau-1}(\Psi_{k+1}-\Psi_k)=\Psi_\tau-\Psi_0.
\end{equation*}
Therefore, we obtain
\begin{align*}
    \E[\Psi_\tau]-\Psi_0
    \le \ &\E\Big[\sum_{k=0}^{\tau-1}\big[\frac1{2L}\int_{A_k}^{A_{k+1}}(s-A_k)\|\nabla f(\xi_{k-1}(s))\|^2\dd s-\frac{(A_{k+1}-A_k)^2}{4L}\|\nabla f(\xi_{k-1}(A_{k+1}))\|^2\big]\Big]\\
    &-\frac{63}{128L}\E\Big[\sum_{k=0}^{\tau-1}(A_{k+1}-A_k)\big(\|\nabla f(x^k)\|^2+\frac{L^2\|p^k\|^2}{A_k^5}\big)\Big],
\end{align*}
which further implies
\begin{align*}
    \E[\Psi_\tau]-\Psi_0 
    \le \ &\frac18\Psi_0+\frac18\E[\Psi_\tau]+\frac1{4L}\E\Big[\sum_{k=0}^{\tau-1}(A_{k+1}-A_k)\big(\|\nabla f(x^k)\|^2+\frac{L^2\|p^k\|^2}{A_k^5}\big)\Big]\\
    &-\frac{63}{128L}\E\sum_{k=0}^{\tau-1}(A_{k+1}-A_k)\Big(\|\nabla f(x^k)\|^2+\frac{L^2\|p^k\|^2}{A_k^5}\Big) \\
    = \ &\frac18\Psi_0+\frac18\E[\Psi_\tau]-\frac{31}{128L}\E\Big[\sum_{k=0}^{\tau-1}(A_{k+1}-A_k)\big(\|\nabla f(x^k)\|^2+\frac{L^2\|p^k\|^2}{A_k^5}\big)\Big].
\end{align*}
by applying Lemma~\ref{lem:random-perturbation}. Therefore, we obtain
\begin{equation*}
    \E[\Psi_\tau]+\frac{31}{112L}\E\sum_{k=0}^{\tau-1}(A_{k+1}-A_k)\Big(\|\nabla f(x^k)\|^2+\frac{L^2\|p^k\|^2}{A_k^5}\Big)\le\frac97\Psi_0.
\end{equation*}
Finally, \eqref{eq:random-local-scales} gives $\frac{31}{112}(A_{k+1}-A_k)\ge\frac14(A_{k+1}-A_k)\ge\frac{A_k^{1/3}}{2048}$. Substituting this bound into the above inequality gives
\begin{equation*}
    \E[\Psi_\tau]+\frac1{2048L}\E\sum_{k=0}^{\tau-1}A_k^{1/3}\Big(\|\nabla f(x^k)\|^2+\frac{L^2\|p^k\|^2}{A_k^5}\Big)\le\frac97\Psi_0,
\end{equation*}
which completes the proof.
\end{proof}

\subsection{Convergence rates}

\begin{proof}[Proof of Theorem~\ref{thm:random-anytime}]
At initialization, $A_0=1$, $p^0=0$, and $\xi_{-1}(s)=x^0$. Smoothness therefore gives $\Psi_0=f(x^0)-f^\star+L\|x^0-x^\star\|^2\le\frac32L\|x^0-x^\star\|^2$. For every deterministic $K\ge1$, by applying Lemma~\ref{lem:random-energy} with $\tau=K$ and dropping the nonnegative sum, we obtain 
\begin{align*}
   \E[\Psi_K]\le\frac97\Psi_0\le\frac{27}{14}L\|x^0-x^\star\|^2 
\end{align*}
Since $x^K=\xi_{K-1}(u_K)$ and $u_K\in[A_K,A_{K+1}]$, integrating $\xi_{K-1}'(s)=p^K/s^3$ gives $\|x^K-\xi_{K-1}(A_K)\|\le\frac{A_{K+1}-A_K}{A_K^3}\|p^K\|$. By using \eqref{eq:Lipschitz-square-root-optimality-gap} and \eqref{eq:random-momentum-control}, we obtain
\begin{align*}
 \sqrt{f(x^K)-f^\star} \le \sqrt{\frac{\Psi_K}{A_K}}+\sqrt{\frac L2}\frac{A_{K+1}-A_K}{A_K^3}\|p^K\| \le \Big(1+\frac{\sqrt2(A_{K+1}-A_K)}{\sqrt{A_K}}\Big)\sqrt{\frac{\Psi_K}{A_K}} \le \Big(1+\frac{\sqrt2}{256}\Big)\sqrt{\frac{\Psi_K}{A_K}},
\end{align*}
where the last step is from \eqref{eq:random-local-scales}. Applying \eqref{eq:random-grid-growth} gives
\begin{equation*}
 f(x^K)-f^\star\le\frac{2\Psi_K}{A_K}\le\frac{2\Psi_K}{(1+K/1024)^{3/2}},
\end{equation*}
which yields
\begin{equation*}
 \E[f(x^K)-f^\star]\le\frac{27L\|x^0-x^\star\|^2}{7(1+K/1024)^{3/2}}\le\frac{4L\|x^0-x^\star\|^2}{(1+K/1024)^{3/2}}\le\frac{131072L\|x^0-x^\star\|^2}{K^{3/2}}.
\end{equation*}

For the almost sure convergence , fix $M>0$ and an integer $N\ge1$. Apply Lemma~\ref{lem:random-energy} to the first index $k\in\{0,\ldots,N\}$ for which $\Psi_k>M$, or to $N$ if there is no such index. This is a bounded stopping time because $\Psi_k$ is $\mathcal G_k$-measurable. The stopped potential is nonnegative and exceeds $M$ on $\{\max_{0\le k\le N}\Psi_k>M\}$. Consequently, we have
\begin{align*}
   M \mathbb P\big(\max_{0\le k\le N}\Psi_k>M\big)\le\frac97\Psi_0.
\end{align*}
Letting $N\to\infty$ gives $\mathbb P(\sup_{k\ge0}\Psi_k>M)\le\frac{9\Psi_0}{7M}$, and letting $M\to\infty$ shows that $\sup_{k\ge0}\Psi_k<\infty$ almost surely. On this probability-one event, we obtain
\begin{equation*}
 \sup_{K\ge1}K^{3/2}[f(x^K)-f^\star]\le2\cdot1024^{3/2}\sup_{k\ge0}\Psi_k<\infty,
\end{equation*}
which completes the proof.
\end{proof}

\section{Concluding Remarks}\label{sec:conclusion}
This paper develops predefined randomized schedules for the heavy-ball method through two levels of randomization. With deterministic time boundaries, both fixed-time and anytime convergence are established with an $\cO(1/K^{4/3})$ convergence rate. Building on the anytime convergence framework, randomizing the boundaries improves the convergence rate to an $\mathcal O(1/K^{3/2})$ last-iterate guarantee, both in expectation and almost surely. Compared with the lower bound $\Omega(1/K^{(1+\sqrt{5})/2}\log K)$ in \cite{ma2026lower}, we conjecture that this lower bound is not tight and that the convergence rate $\mathcal O(1/K^{3/2})$ is already optimal for the heavy-ball method.

\section*{Acknowledgment}
The stepsize and momentum coefficient constructions in this paper were developed by \texttt{GPT-6 Pro} based on problem statements and follow-up instructions provided by the authors. The authors subsequently reorganized and rewrote all proofs in the manuscript and independently verified the correctness of the resulting arguments. The authors take full responsibility for the final manuscript.

\bibliographystyle{plainnat}
\bibliography{ref}

\appendix
\section{Technical Properties of the Time Boundaries}\label{app:time-grid-properties}

\begin{lemma}\label{lem:technical-interval-calculus}
Let $0<a<b$, set $h=b-a$, and let $U\sim\operatorname{Unif}[a,b]$. For any nonnegative integers $r,s$, it holds that
\begin{equation*}
 \E[(U-a)^r(b-U)^s]
 =h^{r+s}\frac{r!s!}{(r+s+1)!}.
\end{equation*}
Moreover, for every $m>0$, we have
\begin{equation}\label{eq:technical-reciprocal-power}
 \frac{a^{-m}-b^{-m}}{m}
 =\int_a^b t^{-m-1}\dd t
 \ge (b-a)b^{-m-1}.
\end{equation}
The same moment identities hold conditionally whenever $a$ and $b$ are fixed under the conditioning and $U$ is conditionally uniform on $[a,b]$.
\end{lemma}
\begin{proof}
By the change of variables $t=(u-a)/h$, we have
\[
 \E[(U-a)^r(b-U)^s]
 =\frac1h\int_a^b(u-a)^r(b-u)^s\dd u
 =h^{r+s}\int_0^1t^r(1-t)^s\dd t.
\]
For nonnegative integers $r,s$, the last integral is the beta integral $\mathrm B(r+1,s+1)=r!s!/(r+s+1)!$, which proves the first identity. For \eqref{eq:technical-reciprocal-power}, integrate the derivative of $-t^{-m}/m$. Since $t\mapsto t^{-m-1}$ is decreasing on $(0,\infty)$, its value on $[a,b]$ is at least $b^{-m-1}$, which gives the lower bound. The conditional statement follows by applying the same deterministic calculation after conditioning.
\end{proof}

\subsection{Deterministic fixed-time coefficients}
The fixed-time construction uses
\[
 A_k=\left(1+\frac{k}{6K^{1/3}}\right)^2,
 \  0\le k\le K.
\]
\begin{lemma}\label{lem:technical-fixed-grid}
For every integer $K\ge1$, the fixed-time grid satisfies $A_k\ge1$ and, with $\theta_k=(A_{k+1}-A_k)/\sqrt{A_k}$, it holds that
\begin{equation}\label{eq:technical-fixed-grid}
 \theta_k
 =\frac{1}{3K^{1/3}}
  +\frac{1}{36K^{2/3}\left(1+\frac{k}{6K^{1/3}}\right)}
 \le\frac{13}{36K^{1/3}}<\frac38,
 \  0\le k<K.
\end{equation}
Consequently, we have
\begin{equation*}
 3\sum_{k=0}^{K-1}\theta_k^3<\frac17.
\end{equation*}
\end{lemma}
\begin{proof}
Set $r_k=1+k/(6K^{1/3})$ and $d=1/(6K^{1/3})$. Then $A_k=r_k^2$, $r_{k+1}=r_k+d$, and
\[
 A_{k+1}-A_k=(r_k+d)^2-r_k^2=2dr_k+d^2.
\]
Dividing by $\sqrt{A_k}=r_k$ gives the equality in \eqref{eq:technical-fixed-grid}. Since $r_k\ge1$ and $K\ge1$, we have
\[
 \theta_k\le\frac{1}{3K^{1/3}}+\frac{1}{36K^{2/3}}
 \le\frac{13}{36K^{1/3}}<\frac38,
\]
which implies
\[
 3\sum_{k=0}^{K-1}\theta_k^3
 \le3K\left(\frac{13}{36K^{1/3}}\right)^3<\frac17.
\]
\end{proof}

\subsection{Deterministic anytime coefficients}
The deterministic anytime construction uses
\[
 A_k=\left(1+\frac{k}{12}\right)^{4/3},
 \  k\ge0.
\]
\begin{lemma}\label{lem:technical-anytime-grid}
For this grid, we have
\begin{equation}\label{eq:technical-anytime-grid}
 \theta_k\coloneqq\frac{A_{k+1}-A_k}{A_k^{1/4}}<\frac18.
\end{equation}
\end{lemma}
\begin{proof}
Let $r_k=1+k/12$, so $A_k=r_k^{4/3}$ and $A_k^{1/4}=r_k^{1/3}$. By the mean value theorem applied to $r\mapsto r^{4/3}$ on $[r_k,r_{k+1}]$, we have
\[
 A_{k+1}-A_k
 \le\frac43 r_{k+1}^{1/3}(r_{k+1}-r_k)
 =\frac19r_{k+1}^{1/3},
\]
which implies
\[
 \theta_k
 \le\frac19\left(\frac{r_{k+1}}{r_k}\right)^{1/3}
 =\frac19\left(1+\frac{1}{12r_k}\right)^{1/3}<\frac18.
\]
\end{proof}

\subsection{Randomized coefficients}
The randomized construction uses
\[
 A_{k+1}=A_k+\frac{A_k^{1/3}}{512}V_k,
 \  V_k\in[1,2],\  A_0=1.
\]
The following lemma collects the grid estimates used in Lemmas~\ref{lem:random-signed-defect} and~\ref{lem:random-perturbation} and the final conversion to deterministic iteration counts.

\begin{lemma}\label{lem:technical-random-grid}
For every realization and every $k\ge0$, we have
\begin{equation}\label{eq:random-local-scales}
 \frac{A_k^{1/3}}{512}\le A_{k+1}-A_k\le\frac{A_k^{1/3}}{256},
\ 
 \frac{A_{k+1}-A_k}{\sqrt{A_k}}\le\frac1{256},
\ 
 \frac{(A_{k+1}-A_k)^3}{A_k}\le\frac1{256^3}.
\end{equation}
Moreover, it holds that
\begin{equation}\label{eq:random-momentum-dissipation}
 \frac14(A_k^{-4}-A_{k+1}^{-4})
 \ge\frac{A_{k+1}-A_k}{2A_k^5},
\end{equation}
\begin{equation}\label{eq:random-scale-change}
 A_{k+1}^{1/3}\le2A_k^{1/3},\ 
 \frac{A_{k+1}^{1/3}-A_k^{1/3}}{512}
 \le\frac{A_{k+1}-A_k}{1536A_k^{2/3}}
 \le\frac2{3\cdot512^2A_k^{1/3}},
\end{equation}
and
\begin{equation}\label{eq:random-grid-growth}
 (1+k/1024)^{3/2}\le A_k\le(1+k/384)^{3/2}.
\end{equation}
\end{lemma}
\begin{proof}
The grid increases from $A_0=1$, and $V_k\in[1,2]$, which gives \eqref{eq:random-local-scales}. In particular, we have
\[
 \frac{A_{k+1}}{A_k}\le1+\frac1{256A_k^{2/3}}\le1+\frac1{256}.
\]
Thus $A_{k+1}^5<2A_k^5$, and \eqref{eq:technical-reciprocal-power} yields
\[
 \frac14(A_k^{-4}-A_{k+1}^{-4})
 \ge\frac{A_{k+1}-A_k}{A_{k+1}^5}
 \ge\frac{A_{k+1}-A_k}{2A_k^5}.
\]
The same ratio bound gives $A_{k+1}^{1/3}\le2A_k^{1/3}$. Note that $A_{k+1}^{1/3}-A_k^{1/3}\le(A_{k+1}-A_k)/(3A_k^{2/3})$ by concavity. Combining this with \eqref{eq:random-local-scales} implies \eqref{eq:random-scale-change}. Finally, the derivative bounds for the concave function $s^{2/3}$ give
\begin{align*}
 A_{k+1}^{2/3}-A_k^{2/3}
 &\le\frac{2(A_{k+1}-A_k)}{3A_k^{1/3}}\le\frac1{384},\\
 A_{k+1}^{2/3}-A_k^{2/3}
 &\ge\frac{2(A_{k+1}-A_k)}{3A_{k+1}^{1/3}}\ge\frac1{1024}.
\end{align*}
Summing from $0$ to $k-1$ yields \eqref{eq:random-grid-growth}.
\end{proof}

\section{Technical Properties of the Kernel}\label{app:kernel-properties}
This appendix supplies the scalar kernel used to prove the cumulative bound in Lemma~\ref{lem:random-perturbation}. Define
\begin{equation}\label{eq:random-kernel-data}
 b(t)=\begin{cases}
 t/2,&0\le t<1,\\
 t-3t^2/4,&1\le t\le2,\\
 0,&\text{otherwise},
 \end{cases}
\end{equation}
and define $\kappa(t)=0$ for $t<0$ and, recursively,
\begin{equation}\label{eq:random-kernel-definition}
 \kappa(t)=b(t)+\int_1^2\kappa(t-v)\dd v,\  t\ge0.
\end{equation}
The right-hand side only uses values at least one unit earlier, so the
recursion uniquely determines a locally bounded kernel on successive
unit intervals.

\begin{lemma}\label{lem:random-kernel}
The kernel in \eqref{eq:random-kernel-definition} satisfies $|\kappa(t)|\le\frac34\,2^{-\lfloor t/3\rfloor}$, $\forall t \ge 0$, and $\int_0^\infty(1+t)^2|\kappa(t)|\dd t<256$.
\end{lemma}
\begin{proof}
Direct substitution in \eqref{eq:random-kernel-definition} gives
\[
 \kappa(t)=\begin{cases}
 t/2,&0\le t<1,\\
 \frac14+\frac t2-\frac{t^2}2,&1\le t\le2.
 \end{cases}
\]
Hence $|\kappa(t)|\le3/4$ on $[0,2]$. For $t>2$, the recursion averages earlier values, so induction over successive unit intervals gives $|\kappa(t)|\le3/4$ for all $t\ge0$. Direct integration of \eqref{eq:random-kernel-data} gives $\int_0^\infty b(t)\dd t=0$. Integrating \eqref{eq:random-kernel-definition} over $[0,t]$ and interchanging the finite integrals therefore yields
\[
 \int_0^1\kappa(t-v)\dd v+\int_1^2(2-v)\kappa(t-v)\dd v=0,
 \ \forall t\ge2.
\]
Replace $t$ by $t-1$ and use the recursion again to obtain
\[
 \kappa(t)=-\int_2^3(3-v)\kappa(t-v)\dd v,\  \forall t\ge3,
\]
which implies $\sup_{t\ge3(n+1)}|\kappa(t)| \le\frac12\sup_{t\ge3n}|\kappa(t)|$, $\forall n\ge0$ since $\int_2^3(3-v)\dd v=1/2$. Iteration proves the stated decay, and summing over intervals of length three gives
\[
 \int_0^\infty(1+t)^2|\kappa(t)|\dd t
 \le\frac34\sum_{n=0}^\infty2^{-n}\int_{3n}^{3n+3}(1+t)^2\dd t
 =\frac34\sum_{n=0}^\infty2^{-n}(27n^2+45n+21)
 =\frac{441}{2}<256.
\]
\end{proof}

\end{document}